\documentclass{article}
\usepackage{amsmath, amssymb, amsthm, graphicx,bm,color,cite,comment}
\usepackage[all]{xy}
\usepackage{titlesec}
\DeclareMathOperator{\gr}{Graph}

\theoremstyle{plain}
\newtheorem{thm}{Theorem}[section]
\newtheorem{lem}[thm]{Lemma}
\newtheorem{cor}[thm]{Corollary}
\newtheorem{prop}[thm]{Proposition}

\newtheorem*{thm*}{Theorem}
\newtheorem*{lem*}{Lemma}
\newtheorem*{cor*}{Corollary}

\theoremstyle{definition}
\newtheorem{dfn}[thm]{Definition}

\newtheorem{rmk}[thm]{Remark}
\newtheorem{hp}[thm]{Hypothesis}
\newtheorem{cond}[thm]{Condition}
\newtheorem*{dfn*}{Definition}

\newtheorem*{hp*}{Hypothesis}

\numberwithin{equation}{section}\numberwithin{figure}{section}

\def\wh{\widehat}
\def\e#1\e{\begin{equation}#1\end{equation}}
\def\iz#1\iz{\begin{itemize}#1\end{itemize}}
\def\ea#1\ea{\e{\begin{split}#1\end{split}}\e}
\def\eq{\eqref}
\def\l{\label}
\def\0{\hspace{0pt}}

\def\ph{\phi}

\def\Uh_#1{\,\widehat{\!U}_{\!#1}}

\def\Ps{\Psi}
\def\Ph{\Phi}

\def\et{\eta}
\def\ps{\psi}

\def\ker{\mathop{\rm ker}}

\def\Im{\mathop{\rm Im}}

\def\ge{\geqslant}
\def\le{\leqslant\nobreak}

\def\cG{{\mathbin{\cal G}}}

\def\cU{{\mathbin{\cal U}}}

\def\={\equiv}

\def\C{{{\mathbb C}}}

\def\R{{\mathbin{\mathbb R}}}
\def\Z{{\mathbin{\mathbb Z}}}

\def\al{\alpha}
\def\be{\beta}
\def\ga{\gamma}
\def\de{\delta}
\def\io{\iota}
\def\ep{\epsilon}
\def\la{\lambda}
\def\ka{\kappa}
\def\th{\theta}
\def\ta{{\tau}}
\def\ze{\zeta}

\def\si{\sigma}
\def\om{\omega}
\def\De{\Delta}

\def\Si{\Sigma}

\def\Th{\Theta}
\def\Om{\Omega}

\def\Up{\Upsilon}

\def\id{{\rm id}}

\def\d{{\rm d}}

\def\ts{\textstyle}

\def\w{\wedge}
\def\-{\setminus}
\def\bu{\bullet}
\def\op{\oplus}
\def\bop{\bigoplus}
\def\ot{\otimes}

\def\ov{\overline}

\def\iy{\infty}

\def\t{\times}
\def\cm{\circ}

\def\nb{\nabla}
\def\sb{\subseteq}
\def\sing{{\rm sing}}

\begin{document}
\date{}
\title{Remarks on the Gluing Theorems for Compact Special Lagrangian Submanifolds with Isolated Conical Singularities}
\author{Yohsuke Imagi}
\maketitle

\begin{abstract}
Let $X$ be a compact special Lagrangian submanifold with isolated conical singularities, and let there exist local smoothings of these singularities. It is then known that under certain hypotheses there exist global smoothings of $X.$ We prove the same results under weaker hypotheses. We prove also that amongst these hypotheses that concerning the cohomology classes of $X$ and the local smoothings is strictly necessary. We prove further that the global smoothings are $C^\iy$ with respect to the relevant parameters.   
\end{abstract}

\section{Introduction}
In this paper we improve upon the gluing construction theorems of Joyce \cite{J3,J4}. Under certain hypotheses Joyce constructs smoothings of compact special Lagrangian submanifolds with isolated conical singularities.

Our first main result may be stated as follows. Let $M$ be a K\"ahler manifold of complex dimension $m>2,$ whose canonical bundle is holomorphically trivial. Let $M$ have a $C^\iy$ family $(\om^s,J^s,\Om^s)_{s\in\R^n}$ of almost Calabi--Yau structures (in the sense of Definition \ref{dfn: almost CY}). Let $X\sb (M;\om^0,J^0,\Om^0)$ be a special Lagrangian submanifold whose closure $\bar X$ is compact and of the form $X\sqcup X^\sing$ for some finite set $X^\sing.$ Suppose that for each $x\in X^\sing$ there exists a unique tangent cone $C_x$ to $X$ at $x$ (which is smooth outside vertex and of multiplicity one). More precisely, we suppose that $X$ approaches with power order at $x$ (as in Definition \ref{dfn: ics}) the cone $C_x.$ Let $L_x$ be a smoothing model for $C_x,$ which is a special Lagrangian submanifold of $T_xM\cong\C^m$ which approaches (with order $\le0$) at infinity the cone $C_x.$ Recall from \cite{J3,J4} that for each $x\in X^\sing$ there are two cohomology classes $Y(L_x)\in H^1(C_x,\R)$ and $Z(L_x)\in H^{m-1}(C_x,\R);$ for the definition see also Definition \ref{dfn: ACL}. We prove 

\begin{thm}\l{1.1}
Let $X$ be connected. Then there exists $t_0>0$ such that we can glue together $X$ and $(L_x)_{x\in X^\sing}$ into a family $(N^t)_{t\in(0,t_0)}$ of compact special Lagrangian submanifolds of $(M;J^0,\om^0,\Om^0).$ Here $t$ is the parameter with which we rescale each $L_x.$
\end{thm}

This is proved in \cite[Theorem 6.13]{J4} under the additional hypothesis $m<6.$ We get rid of this as follows. We begin with the same approximate solutions as in \cite{J3,J4} which we make from $X$ and $\bigsqcup_{x\in  X^\sing}L_x$ using partitions of unity (not doing analysis yet). We then perturb these to the true solutions, not by the method of \cite{J3,J4} but by the method of \cite{I2,P}. The advantage of the latter is that the two building blocks $X$ and $\bigsqcup_{x\in  X^\sing}L_x$ play symmetric r\^oles. This happens to other gluing methods including those in Yang--Mills gauge theory. For instance, if we want to glue together the model ASD connections over $\R^4$ and the flat connection over a compact $4$-manifold $M$ then we can take the actual building blocks to be $S^4=\R^4\sqcup\{\iy\}$ and $M$ (as in \cite[\S7.2]{DK}) whose r\^oles are symmetric.

More specifically, in \cite{J3,J4} the dimension hypothesis $m<6$ is crucial to an $L^{\frac{2m}{m+2}}$ estimate for $\Im\Om|_N;$ for more details see Remark \ref{worse rmk}. As in Remark \ref{exp} the $L^{\frac{2m}{m+2}}$ estimate is crucial to \cite[Theorem 5.3]{J3}. In our treatment however we do not use \cite[Theorem 5.3]{J3} to prove Theorem \ref{1.1} above. The analogue of the $L^{\frac{2m}{m+2}}$ estimate is Proposition \ref{near prop}.

The dimension hypothesis appears also in \cite[Theorem 6.12]{J4} in which $X$ is no longer connected. We prove again that the dimension hypothesis is unnecessary in these circumstances; that is, the following holds.
\begin{thm}\l{1.2}
Let $(M;J^0,\om^0,\Om^0),X, (C_x)_{x\in X^\sing}$ and $(L_x)_{x\in X^\sing}$ be as in Theorem \ref{1.1} except that $X$ should be connected. Denote by $N$ the compact manifold without boundary which we obtain from $X$ and $\bigsqcup_{x\in X^\sing}L_x$ after we identify their common ends. Let $N$ be connected. Let $C_{xj}$ be as in Definition \ref{dfn: ics} and define $\ps^0:M\to(0,\iy)$ by \eq{ps} with $(\om^0,J^0,\Om^0)$ in place of $(\om,J,\Om).$ Suppose that every connected component $Y\sb X$ satisfies the equation
\e\l{coh1}
\sum_{C_{xj}\sb Y}\ps^0(x)^m[C_{xj}\cap S^{2m-1}]\cdot Z(L_x)=0.
\e
Then we can glue together $X$ and $(L_x)_{x\in X^\sing}$ as in Theorem \ref{1.1}.
\end{thm}
\begin{rmk}\l{1.3}
The definition of $Z(L_x)$ implies readily that the sum over $Y$ of the left-hand side of \eq{coh1} is equal to zero. So \eq{coh1} holds automatically for $X$ connected as in Theorem \ref{1.1}.
\end{rmk}

The reason that the dimension hypothesis is unnecessary for Theorem \ref{1.2} is essentially the same as that for Theorem \ref{1.1}. There is however a new issue in the proof of Theorem \ref{1.2}. In the proof of Theorem \ref{1.1} the perturbations of the first approximate solutions are the solutions we want. In the proof of Theorem \ref{1.2} however the corresponding perturbations are not the true solutions yet. These are called the second approximate solutions. 

Here is why they are not the solutions yet. When we perturb the first approximate solutions, we linearize the differential operator which defines special Lagrangian submanifolds. But as $X$ is not connected, the linearized operator is in general not surjective.

This is why we need the condition \eq{coh1}; roughly speaking, \eq{coh1} means that the linearized operator is surjective. More precisely, \eq{coh1} implies that the second approximate solutions satisfy much better estimates than the first approximate solutions. As a consequence, the perturbation theorem \cite[Theorem 5.3]{J3} applies to the second approximate solutions and we can perturb them to the true solutions. So in the proof of Theorem \ref{1.2} we use both the perturbation method in \cite{J3,J4} and that of \cite{I2,P}. 

Our next main result is about the most general version \cite[Theorem 8.9]{J4} of the gluing theorems in \cite{J3,J4}. We get rid of the dimension hypothesis and another hypothesis on the family $(Z(L_x))_{x\in X^\sing};$ that is, the following holds.
\begin{thm}\l{1.4}
Let $(M;J^s,\om^s,\Om^s)_{s\in\R^n},X, (C_x)_{x\in X^\sing}$ and $(L_x)_{x\in X^\sing}$ be as at the beginning of this section. Let $N$ be as in Theorem \ref{1.1}. Define $\ps^s:M\to(0,\iy)$ by \eq{ps} with $(\om^s,J^s,\Om^s)$ in place of $(\om,J,\Om).$ Let $\cG\sb \R^n\times (0,\iy)$ be the set of $(s,t)$ such that 
\iz
\item[\bf(i)]
the image of $[\om^s]\in H^2(M)$ under the natural map $H^2(M)\to H^2(\bar X)\cong H^2_c(X)$ agrees with the image of $\bop_{x\in  X^\sing}t^2Y(L_x)\in \bop_{x\in X^\sing}H^1(C_x)$ under the natural map $\bop_{ X^\sing}H^1(C_x)\to H^2_c(X);$ and
\item[\bf(ii)]
every connected component $Y\sb X$ satisfies the equation
\e\l{coh2}
\sum_{C_{xj}\sb Y}t^m\ps^s(x)^m[C_{xj}\cap S^{2m-1}]\cdot Z(L_x)=[\bar Y]\cdot [\Im\Om^s].
\e
\iz
Then for every $\th>2$ there exist $t_0>0$ and a family $(N^{st}\sb M:(s,t)\in\cG,|s|< t^\th<t_0^\th)$ of compact special Lagrangian submanifolds of $(M;J^s,\om^s,\Om^s).$
\end{thm}
\begin{rmk}
The condition (i) above is equivalent to saying that $N^{st}\sb (M,\om^s)$ is a Lagrangian submanifold. We discuss more about (ii) in what follows.
\end{rmk}

In \cite[(135) and (136)]{J4} it is supposed that both sides of \eq{coh2} vanish. If $X$ is connected then as in Remark \ref{1.3} the left-hand side of \eq{coh2} vanishes automatically and the equation \eq{coh2} is equivalent to saying that both sides of it vanish. But for $X$ disconnected it is artificial to suppose that both sides of \eq{coh2} vanish. We prove indeed that there is a slightly weaker version of \eq{coh2} which is a necessary condition for the gluing process to be possible. For the more precise statements see \eq{eq for Om} and Theorem \ref{obst thm}. Our results imply also that \eq{coh2} or its weaker version \eq{eq for Om} is sufficient for the gluing process to be possible. The proof of this is essentially the same as that of Theorem \ref{1.2}.

We begin in \S\ref{st sec} with the more precise statements of our results. In \S\ref{obst sec} we prove that the weaker equation \eq{eq for Om} is necessary for the gluing process to be possible. In \S\S\ref{app sec}--\ref{Proof of 1} we prove also that it is a sufficient condition.

In \S\ref{app sec} we produce the first approximate solutions, which is essentially the same as in \cite{J3,J4}. In \S\ref{Proof of 1--3} we estimate how close these are to being special Lagrangian. As we drop the dimension hypothesis, the estimates we get are weaker than those in \cite{J3,J4}. But as we have explained above, we follow the method of \cite{I2,P} rather than that of \cite{J3,J4}; and the weaker estimates will do for our purposes.  

In \S\ref{pert2} we prove the quadratic estimates for the non-linear part of the special Lagrangian equations, which is not very different from those in \cite{J3,J4}. In \S\ref{pert3} we prove that the linearized operators are uniformly invertible, which is essentially the same as in \cite{I2,P}. In \S\ref{pert4} we produce the second approximate solutions. If $X$ is connected then these solutions are already the true solutions as we have explained above. In \S\ref{Proof of 1} we perturb the second approximate solutions to the true solutions. 

In \S\S\ref{sect10} and \ref{sect11} we prove that the special Lagrangian submanifold $N^{st}$ produced in Theorem \ref{1.4} depends smoothly on $s,t.$ The conditions (i),(ii) in Theorem \ref{1.4} imply that the set of $(s,t)$ which parametrizes $N^{st}$ is a closed set in the $(s,t)$ space. Smooth with respect to $s,t$ means that every $(s,t)$ has an open neighbourhood to which $N^{st}$ extend smoothly. The extended $N^{st}$ is no longer a special Lagrangian submanifold. In other words, we introduce an auxiliary term $\Xi^{st}$ in the equation we solve; for the definition of $\Xi^{st}$ see \eq{Xi}. For $(s,t)$ satisfying (i),(ii) of Theorem \ref{1.2} we have $\Xi^{st}=0$ and the equation we solve is that of special Lagrangian submanifolds. Otherwise the equation has no such geometric meaning but the parameter space is now an open subset of the $(s,t)$ space.  

The second approximate solutions and the true solutions both extend to this open set of $(s,t).$ They are produced by the contraction mapping theorem. So there are sequences of functions whose limits are the solutions. We prove in Theorem \ref{3} that these functions depend smoothly on $s,t.$ Differentiating them (both in the original manifold direction and in the $(s,t)$ direction) we get countably many sequences of functions. We prove by inductions that these functions depend smoothly on $s,t.$ They are in fact Cauchy sequences such that the limit functions will depend smoothly on $s,t.$ The details are given in \S\ref{sect10}. Using its results we conclude in \S\ref{sect11} that the second approximate solutions and the true solutions both depend smoothly on $s,t.$

{\bf Acknowledgements.} The author was supported by the grant 21K13788 of the Japan Society for the Promotion of Science.

\section{Statements of the Results}\l{st sec}
We begin by making the definitions we will use to state our results. We define almost Calabi--Yau $m$-folds in the first place.
\begin{dfn}\l{dfn: almost CY}
Let $M$ be a manifold of even dimension $2m.$ An {\it almost Calabi--Yau} structure on $M$ is the data $(\om,J,\Om)$ such that $(M;\om,J)$ is a K\"ahler manifold with K\"ahler form $\om$ and complex structure $J,$ and such that $\Om$ is a nowhere-vanishing holomorphic $(m,0)$ form on the complex manifold $(M,J).$ We call $(M;\om,J,\Om)$ an {\it almost Calabi--Yau $m$-fold.} Notice that there is a $C^\iy$ function $\ps:M\to(0,\iy)$ defined by
\e\l{ps} \ps^{2m}\om^m/m!=(i/2)^m(-1)^{m(m-1)/2}\Om\w\ov\Om.\e
We call $(\om,J,\Om)$ a {\it Calabi--Yau} structure, without the word almost, if $\ps:M\to(0,\iy)$ is constant. In this case $\om$ is a Ricci-flat K\"ahler form on $(M,J).$
\end{dfn}

We define special Lagrangian submanifolds of almost Calabi--Yau $m$-folds. 
\begin{dfn}
Let $(M;\om,J,\Om)$ be an almost Calabi--Yau $m$-fold and $L\sb (M,\om)$ a Lagrangian submanifold. We say that $L$ is {\it special} if $\Im\Om|_L=0.$
\end{dfn}
We define a Calabi--Yau structure $(\om',J',\Om')$ on $\C^m$ and special Lagrangian cones in $(\C^,\om',J',\Om').$
\begin{dfn}\l{dfn C^m}
Let $m\ge1$ be an integer and $z_1,\dots,z_m$ the complex coordinates of $\C^m.$ Denote by $J'$ the complex structure of $\C^m$ and define on $\C^m$ a Calabi--Yau structure $(\om',J',\Om')$ by $\om':=\frac{i}{2}({\rm d}z_1\w {\rm d}\bar z_1+\dots+{\rm d}z_m\w {\rm d}\bar z_m)$ and $\Om':={\rm d}z_1\w\dots\w {\rm d}z_m.$ 

Define $(0,\iy)\times\C^m\to\C^m$ by $(t,z)\mapsto tz,$ by which the multiplicative group $(0,\iy)$ acts upon $\C^m.$ For $t\in(0,\iy)$ we denote for short by $t:\C^m\to\C^m$ the map $z\mapsto t z.$ A {\it special Lagrangian cone} in $(\C^m;\om',J',\Om')$ is a special Lagrangian submanifold $C\sb (\C^m;\om',J',\Om')$ invariant under the $(0,\iy)$ action and such that $C\cap S^{2m-1}$ is a compact $(m-1)$ dimensional submanifold of the unit sphere $S^{2m-1}\sb\C^m.$ The {\it cone metric} on $C$ is the Riemannian metric induced from the embedding $C\sb\C^m.$
\end{dfn}
\begin{rmk}
Our $C$ corresponds to Joyce's $C'$ in \cite{J1,J2,J3,J4,J5}, and our $\bar C=C\sqcup\{0\}$ to Joyce's $C.$  We change the notation because we refer more often to $C$ than to $\bar C.$
\end{rmk}

We define an asymptotically conical special Lagrangian submanifold $L\sb\C^m$ and its cohomology classes $Y(L),Z(L).$
\begin{dfn}\l{dfn: ACL}
Let $C\sb(\C^m;\om',J',\Om')$ be a special Lagrangian cone and $L\sb (\C^m;\om',J',\Om')$ a properly-embedded special Lagrangian submanifold. We say that $L$ approaches {\it with order $\la<2$ at infinity} the cone $C$ if there exist a compact set $K\sb \C^m$ and a diffeomorphism $\ph:C\-K\cong L\-K$ such that if we denote by $\io:C\to\C^m$ the inclusion map then for $k=0,1,2,\dots$ we have $|\nb^k(\ph-\io)|=O(r^{\la-1-k})$ where $|\,\,|$ and $\nb$ are computed with respect to the cone metric on $C.$ 

Let $C\sb(\C^m;\om',J',\Om')$ be a special Lagrangian cone and $L\sb (\C^m;\om',J',\Om')$ a properly-embedded special Lagrangian submanifold which approaches with order $<2$ at infinity the cone $C.$ Using the homology groups with coefficient field $\R$ we make the following definition: $Y(L)\in H^1(C)$ and $Z(L)\in H^{m-1}(C)$ are the respective images of the relative de Rham classes
\[ [\om'|_L]\in H^2(\C^m,L)\cong H^1(L)\text{ and } [\Im\Om'|_L]\in H^m(\C^m,L)\cong H^{m-1}(L)\]
under the natural maps $H^1(L)\to H^1(C)$ and $H^{m-1}(L)\to H^{m-1}(C).$ 
\end{dfn}

We define compact special Lagrangian submanifolds with isolated conical singularities. 
\begin{dfn}\l{dfn: ics}
Let $(M;\om,J,\Om)$ be an almost Calabi--Yau $m$-fold and define $\ps:M\to(0,\iy)$ by \eq{ps}. Let $X\sb (M;\om,J,\Om)$ be a special Lagrangian submanifold with compact closure $\bar X=X\sqcup X^\sing$ for some finite set $X^\sing.$ For each $x\in X^\sing$ choose a $\C$-vector space isomorphism $(T_xM,J|_x)\cong(\C^m,J')$ under which $(\om|_x,\Om|_x)$ corresponds to $(\om',\Om').$ We call $\bar X$ a compact special Lagrangian submanifold {\it with isolated conical singularities} if for each $x\in X^\sing$ there exist a special Lagrangian cone $C_x\sb (\C^m,\om',J',\Om'),$ an open neighbourhood $U_x$ of $0\in\C^m,$ and a Darboux embedding $\Up_x:U_x\to M$ such that $\Up_x(0)=x,$ $\Up_x^*\Om=\ps(x)^m\Om',$ $\Up_x^*\om=\om'$ and the following holds: there exist $\mu>2$ and a diffeomorphism $\ph:C_x\cap U_x\to \Up_x^{-1}(X)$ such that if we denote by $\io:C_x\to \C^m$ the inclusion map then for $k=0,1,2,\dots$ we have $|\nb^k(\ph-\io)|=O(r^{\mu-1-k})$ where $|\,\,|,\nb$ are computed with respect to the cone metric on $C_x.$ 

We call $C_x$ the {\it tangent cone} to $X$ at $x.$ We introduce the following notation for short: for each connected component $Y\sb X$ and each connected component $C_{xj}\sb C_x$ we write $C_{xj}\sb Y$ if $\Up_x(C_{xj}\cap U_x)\sb Y.$ 
\end{dfn}
\begin{rmk}
Denote by $H^*_c(X):=H^*_c(X,\R)$ the compact-support cohomology group. Since $X$ is outside a compact set diffeomorphic to $\bigsqcup_{x\in X^\sing}C_x$ we get a long exact sequence $\dots \to H^*_c(X)\to H^*(X)\to \bop_{x\in X^\sing} H^*(C_x)\to\cdots.$ Since $m\ge 3$ it follows also that for $q=2,3,4,\dots$ there are natural $\R$-vector space isomorphisms $H^q_c(X)\cong H^q(\bar X, X^\sing)\cong H^q(\bar X).$
\end{rmk}

We state now the hypothesis under which we carry out the gluing process.
\begin{hp}\l{hp}
Let $m\ge3$ be an integer and $M$ a $C^\iy$ manifold of dimension $2m.$ Let $n\ge0$ be an integer and $(\om^s,J^s,\Om^s)_{s\in\R^n}$ a $C^\iy$ family of almost Calabi--Yau structures on $M.$ Let $\bar X\sb (M;\om^0,J^0,\Om^0)$ be a compact special Lagrangian submanifold with isolated conical singularities. Suppose that for each $x\in X^\sing$ there exists a properly-embedded special Lagrangian submanifold $L_x\sb  (\C^m;\om',J',\Om')$ which approaches with order $\le0$ at infinity the cone $C_x.$
\end{hp}

We make a definition we will use shortly.
\begin{dfn}\label{smoothness}
Let $n\ge0$ be an integer and $\cG\sb\R^n\times(0,\iy)$ a subset whose closure in $\R^n\times[0,\iy)$ contains $(0,0).$ We say then that a function $\be:\cG\to\R$ is {\it smooth} if every point of $\cG$ has an open neighbourhood in $\R^n\times(0,\iy)$ to which $\be$ extends smoothly. In the same way we define a smooth function $\cG\times N\to M$ for $M,N$ manifolds. Suppose now that $(N^{st}\sb M)_{(s,t)\in\cG}$ a family of submanifolds, each $N^{st}$ diffeomorphic to a common manifold $N.$ We say then that the family is {\it smooth} if there exists a smooth function $\cG\times N\to M$ whose restriction to $\{(s,t)\}\times N$ defines the submanifold $N^{st}\sb M.$
\end{dfn}

We state the condition which will play the central r\^ole in stating our main results. 
\begin{cond}\l{cond}
Under Hypothesis \ref{hp} there exist a subset $\cG\sb\R^n\times(0,\iy)$ whose closure in $\R^n\times[0,\iy)$ contains $(0,0),$ and a smooth family $(N^{st}\sb  M)_{(s,t)\in\cG}$ of compact submanifolds with the following three properties.
\iz
\item[\bf(a)] Each $N^{st}\sb (M;\om^s,J^s,\Om^s)$ is a special Lagrangian submanifold.
\item[\bf(b)] As $(s,t)$ tends to $(0,0)$ the family $N^{st}$ converges in the sense of geometric measure theory to $\bar X.$ More precisely, the multiplicity-one varifolds supported on $N^{st}$ converge to the multiplicity-one varifold supported on $\bar X.$
\item[\bf(c)]
There exist constants $c,\ep>0$ independent of $s,t$ and such that the following holds. For $x\in X^\sing$ let $U_x\sb M$ be as in Definition \ref{dfn: ics}, which we identify with the corresponding neighbourhood of $0\in\C^m.$ Let $K_x\sb \C^m$ be a compact neighbourhood of $0$ and denote by $\io:L_x\cap K_x\to\C^m$ the inclusion map. Then for every $(s,t)$ close enough to $(0,0)$ there exists a diffeomorphism $\ph:L_x\cap K_x\to t^{-1}(N^{st}\cap tK_x)$ such that for $k=0,1,2,\dots$ we have $|\nb^k(\ph-\io)|\le c(|s|^\ep+t^\ep)$ where $|\,\,|,\nb$ are computed pointwise with respect to the metric on $L_x\cap K_x$ induced from the Euclidean metric of $K_x\sb\C^m.$
\iz
\end{cond}
\begin{rmk}
Suppose that these (a)--(c) hold. Combining them with known results we can control $N^{st}$ as follows. By Allard's regularity theorem it follows from (a),(b) above that the family $N^{st}\sb  M$ converges smoothly in every compact set to $X$ as $(s,t)$ tends to $(0,0).$ So there exists a neighbourhood $\bigsqcup_{x\in X^\sing}U_x$ of $X^\sing\sb M$ such that $N^{st}\-\bigsqcup_{x\in X^\sing}U_x$ is $C^\iy$ close to $X\-\bigsqcup_{x\in X^\sing}U_x.$ On the other hand, by (c) above we can control $N^{st}$ also in $\bigsqcup_{x\in X^\sing}tK_x.$ By the proof of \cite[Theorem 2.2]{I0} we can in fact control $N^{st}$ everywhere; that is, the annulus region $\bigsqcup_{x\in X^\sing}(N^{st}\cap U_x\-tK_x)$ is $C^\iy$ close to $\bigsqcup_{x\in X^\sing} (C_x\cap U_x\-tK_x)$ where each $C_x$ is embedded in $X$ as in Definition \ref{dfn: ics}.
\end{rmk}

We make now a definition and state then the result that the cohomology class equations are strictly necessary for the gluing process to be possible.
\begin{dfn}
Let $n\ge0$ be an integer and $\cG\sb\R^n\times(0,\iy)$ a subset whose closure in $\R^n\times[0,\iy)$ contains $(0,0).$ We say then that a function $\be:\cG\to\R$ {\it decays with power decay at $(0,0)\in\ov\cG$} if there exist $\ep,c>0$ such that for every $(s,t)\in \cG$ we have $|\be(s,t)|\le c(|s|^\ep+t^\ep).$
\end{dfn}
\begin{thm}\l{obst thm}
If Hypothesis \ref{hp} and Condition \ref{cond} hold then there exist smooth functions $\be_x,\ga_x:\cG\to\R$ decaying with power order at $(0,0)\in\ov\cG$ and such that for every $(s,t)\in\cG$ the following hold.
\iz
\item[\bf(i)]
The image of $[\om^s]\in H^2(M)$ under the natural map $H^2(M)\to H^2(\bar X)\cong H^2_c(X)$ agrees with the image of $\bop_{x\in  X^\sing}t^2(1+\be_x)Y(L_x)\in \bop_{x\in X^\sing}H^1(C_x)$ under the natural map $\bop_{ X^\sing}H^1(C_x)\to H^2_c(X).$
\item[\bf(ii)]
Define $\ps^s:M\to(0,\iy)$ by \eq{ps} with $(\om^s,J^s,\Om^s)$ in place of $(\om,J,\Om).$ Then for each connected component $Y\sb X$ we have
\e\l{eq for Om}
\sum_{C_{xj}\sb Y}t^m\ps^s(x)^m(1+\ga_x)[C_{xj}\cap S^{2m-1}]\cdot Z(L_x)=[\bar Y]\cdot [\Im\Om^s]
\e
where the sum on the left-hand side is taken over connected components $C_{xj}\sb C_x$ lying in $Y$ after embedded in $M,$ and the right-hand side is the pairing of $[\bar Y]\in H_m(M)$ and $[\Im\Om^s]\in H^m(M).$
\iz
\end{thm}

We finally state the gluing theorem under the weaker hypotheses.
\begin{thm}\l{1}
Let Hypothesis \ref{hp} hold. Denote by $N$ the compact manifold without boundary which we obtain from $X$ and $
\bigsqcup_{x\in X^\sing}L_x$ after we identify their common ends. Let $N$ be connected.

Let $\cG\sb\R^n\times(0,\iy)$ be a subset whose closure in $\R^n\times[0,\iy)$ contains $(0,0).$ Suppose that for each $x\in X^\sing$ we are given two smooth functions $\be_x,\ga_x:\cG\to\R$ decaying with power order at $(0,0)\in\ov\cG$ and such that every $(s,t)\in\cG$ satisfies {\rm (i),(ii)} of Theorem \ref{obst thm}. Then for every $\th>2$ there exist $t_0>0$ and a family $(N^{st}\sb M:(s,t)\in\cG,|s|< t^\th<t_0^\th)$ of compact submanifolds satisfying {\rm(a)--(c)} of Condition \ref{cond}.
\end{thm}

The submanifold $N^{st}$ may be written as the image of some embedding $(\io^{st}:N\to M)_{(s,t)\in\cG}.$ In this expression $N,M$ are independent of $s,t;$ it is only $\io^{st}$ that depends on $s,t.$ We prove that $\io^{st}$ is $C^\iy$ with respect to $s,t$ in the sense of Definition \ref{smoothness}.

\begin{thm}\l{smooth dependence}
In the circumstances of Theorem \ref{1}  there exists $t_1\in(0,t_0)$ such that if $t<t_1$ then $\io^{st}$ is $C^\iy$ with respect to $s,t.$ 
\end{thm} 

\section{Proof of Theorem \ref{obst thm}}\l{obst sec}
Let Hypothesis \ref{hp} and Condition \ref{cond} hold. For $x\in X^\sing$ choose as in \cite[Theorem 3.8]{J1} a Weinstein neighbourhood of the tangent cone $C_x\sb  \C^m$ and a symplectic embedding of it into the cotangent bundle $T^*C_x.$ The difference from the ordinary Weinstein theorem is the group action by the multiplicative group $(0,\iy).$ This acts on $C_x$ and on $\C^m,$ under which the inclusion map $C_x\to\C^m$ is equivariant. The action $(0,\iy)\curvearrowright C_x$ induces also an $(0,\iy)$ action on $T^*C_x;$ and in particular, if $\al$ is a $1$-form on $C_x$ then $t(\gr\al)=\gr(t^2t_*\al)\sb T^*C_x.$ We require that the neighbourhoods of $C_x\in \C^m$ and $C_x\sb T^*C_x$ are both $(0,\iy)$ invariant and that the diffeomorphism betweem these is $(0,\iy)$ equivariant. 

For $x\in X^\sing$ let $\Up_x:U_x\to M$ be as in Definition \ref{dfn: ics} with $(\om^0,J^0,\Om^0)$ in place of $(\om,J,\Om).$ Following \cite[Theorem 4.4]{J1} write $X\cap U_x$ as the graph of a $1$-form $\et_x$ over $C_x\cap U_x$ in its Weinstein neighbourhood. It is proved in \cite[Lemma 4.5]{J1} that the $1$-form $\et_x$ is exact. Take a Weinstein neighbourhood of $X\cap U_x$ such that the following holds.
\begin{equation}\l{Wein}\parbox{10cm}{
Denote by $\ph:C_x\cap U_x\to X\cap U_x$ the diffeomorphism assigning to each point $y\in C_x\cap U_x$ the unique point of $X\cap U_x=\gr\et_x$ in the fibre over $y$ (in the Weinstein neighbourhood of $C_x\cap U_x).$ Then for every $1$-form $\al$ on $X\cap U_x$ its graph in the Weinstein neighbourhood of $X\cap U_x$ is equal to the graph of $\ph^*\al+\et_x$ in the Weinstein neighbourhood of $C_x\cap U_x.$ 
 }\end{equation}
Following \cite[Theorem 4.8]{J1} we define a Weinstein neighbourhood of $X\sb M.$ 
\begin{dfn}\l{dfn: Wein X}
Define $\ps^s:M\to(0,\iy)$ by \eq{ps} with $(\om^s,J^s,\Om^s)$ in place of $(\om,J,\Om)$ (which is the same as in Theorem \ref{obst thm} (ii)). For each $x\in X^\sing$ choose a $C^\iy$ family of $\C$-vector space isomorphisms $(T_xM,J^s|_x)\cong(\C^m,J')$ under which every $(\om^s|_x,\Om^s|_x)$ corresponds to $(\om',\Om').$ By a family version of Darboux's theorem we can find, making $U_x$ smaller if we need, a Darboux embedding $\Up^s_x:U_x\to M$ such that $\Up^s_x(0)=x,$ $\Up^{s*}_x\Om^s=\ps^s(x)^m\Om'$ and $\Up^{s*}_x\om^s=\om'.$ We often identify $U_x\sb\C^m$ with its image $\Up^s_x(U_x)\sb M.$

Choose a finite-dimensional $\R$-vector space $V$ consisting of closed $2$-forms on $X$ supported in the compact set $X\-\bigsqcup_{x\in X^\sing}U_x$ and such that the natural projection $V\to H_c^2(X)$ is an $\R$-vector space isomorphism. Denote by $\nu^s\in V$ the element whose compact-support cohomology class is equal to the image of $[\om^s]$ under $H^2(M)\to H^2(\bar X)\cong H_c^2(X).$ Recall then from \cite[Theorem 4.8]{J1} that there exists an $s$-dependent neighbourhood $W$ of $X\sb  M$ such that for each $x\in X^\sing$ the pre-image $(\Up^s_x)^{-1}(W)$ agrees with the Weinstein neighbourhood of $X\cap U_x\sb\C^m,$ and $W$ is diffeomorphic to a neighbourhood of the zero-section $X\sb T^*X$  so that if we denote by $\wh\om$ the standard symplectic form of the cotangent bundle $T^*X$ then the K\"ahler form $\om^s$ corresponds to $\wh\om+\nu^s.$ We call $W$ the {\it Weinstein} neighbourhood of $X\sb(M,\om^s).$
\end{dfn}

Consider now the smooth model $L_x\sb\C^m$ for $x\in X^\sing.$ 
\begin{dfn}\l{dfn: L_x}
Following \cite[Theorem 7.4]{J1} let $K_x\sb \C^m$ be a compact neighbourhood of $0\in\C^m$ and write $L_x\- K_x$ as the graph of a $1$-form $\chi_x$ over $C_x\-K_x$ in the Weinstein neighbourhood of $C_x.$ It is proved in \cite[Proposition 7.6]{J1} that the cohomology class $[\chi_x]$ is equal to $Y(L_x).$ Take a Weinstein neighbourhood of $L_x\- K_x$ such that \eq{Wein} holds with $L_x\-K_x$ in place of $X\cap U_x$ and with $\chi_x$ in place of $\et_x.$ Following \cite[Theorem 7.5]{J1} we take a Weinstein neighbourhood of $L_x\sb(\C^m,\om')$ which extends that of $L_x\- K_x$ which we have just defined.
\end{dfn}

Denote by $K_x^\cm$ the interior of the compact subset $K_x\sb\C^m.$ This is by definition an open set in $\C^m.$ Making $K_x$ larger if we need, suppose that $K_x^\cm$ is contractible and $L_x\cap K_x^\cm$ diffeomorphic to $L_x.$ Define then the cohomology class $Y(L_x\cap K_x)\in H^1(C_x)$ to be image of the relative de Rham cohomology class $[\om']\in H^2(K_x^\cm,L_x\cap K_x^\cm)$ under the composite of the natural maps $H^2(K_x^\cm,L_x\cap K_x^\cm)\cong H^1(L_x\cap K_x^\cm)\cong H^1(L_x)\to H^1(C_x\cap S^{2m-1})\cong H^1(C_x).$ 

By hypothesis we are given $N^{st}$ as in Condition \ref{cond}. As $t^{-1}(N^{st}\cap tK_x^\cm)\sb K_x^\cm$ is a Lagrangian submanifold we can assign to it in the same way as above the cohomology class $Y(t^{-1}(N^{st}\cap tK_x^\cm)).$ Part (c) of Condition \ref{cond} implies then that $t^{-1}(N^{st}\cap tK_x)$ converges with power orders to $L_x\cap K_x$ so that \e\l{Y}
Y(t^{-1}(N^{st}\cap tK_x))=(1+\be_x)Y(L_x\cap K_x)
\e
for some $\be_x=\be_x(s,t)$ decaying with power orders. But by definition $Y(L_x\cap K_x)$ and $Y(L_x)$ are equal. The equation \eq{Y} implies therefore (i) of Theorem \ref{obst thm}.

For each connected component $Y\sb  X$ there is a unique connected component of $N^{st}\-\bigsqcup_{x\in X^\sing}tK_x$ that is $C^1$ close to $Y\-\bigsqcup_{x\in X^\sing} tK_x,$ which we denote by $Y^{st}.$ This $Y^{st}$ is a manifold with boundary, with $\partial Y^{st}\sb \bigsqcup_{x\in X^\sing} tK_x.$ For $x\in X^\sing$ let $\Si_{xj}\sb t^{-1}(\partial Y^{st}\cap tK_x)$ be a connected component, which is a compact $m-1$ manifold diffeomorphic to some connected component $C_{xj}\cap S^{2m-1}\sb C_x\cap S^{2m-1}.$ Denote by $0*\Si_{xj}\sb \C^m$ the union of the line segments between the origin $0$ and the points of $\Si_{xj}.$ The union $Y^{st}\cup \bigsqcup_{C_{xj}\sb Y} t(0*\Si_{xj})$ defines in $M$ an $m$-cycle homologous to $\bar Y.$ Integrating $\Im\Om^s$ over this and using $\Im\Om^s|_{Y^{st}}=0$ we get
\e\l{obst2} [\bar Y]\cdot [\Im\Om^s]=\sum_{C_{xj}\sb Y}\int_{t(0*\Si_{xj})}\Im\Om^s.\e
On the other hand, for each $x\in X^\sing$ we have
\e\l{obst3}\int_{t(0*\Si_{xj})}\Im\Om^s=t^m\ps^s(x)^m\int_{0*\Si_{xj}}\Im t^{-m}\ps^s(x)^{-m}t^*\Om^s.\e
Notice that $t^{-m}\ps^s(x)^{-m}t^*\Om^s$ depends smoothly on $(s,t)$ including $(0,0),$ at which it is equal to $\Om',$ so that
\e \l{obst4}
\int_{0*\Si_{xj}}\Im t^{-m}\ps^s(x)^{-m}t^*\Om^s= \int_{0*\Si_{xj}}\Im\Om'+O(|s|+t).\e
Let $\Ps$ be an $m-1$ form on $\C^m$ with $\d\Ps=\Im\Om'.$ Then 
\e\l{obst5}
\int_{0*\Si_{xj}}\Im\Om'=\int_{\Si_{xj}}\Ps.
\e
Notice that $\Si_{xj}$ converges, as $(s,t)$ tends to $(0,0),$ to some connected component of $\partial(L_x\cap K_x)$ which we denote by $\Si^0_{xj}.$ Since $\Si_{xj}$ converges with power orders (as in (c) of Condition \ref{cond}) to $\Si^0_{xj}$ it follows that 
\e \l{obst6}
\int_{\Si_{xj}}\Ps=\int_{\Si^0_{xj}}\Ps+O(|s|^\ep+t^\ep)
\e
for some $\ep>0$ independent of $s,t.$ More precisely, making $K_x$ larger if we need, we can suppose that the $1$-form $\chi_x$ in Definition \ref{dfn: L_x} is defined also near $\partial (C_x\cap K_x).$ Define a diffeomorphism $\ph:\partial(C_x\cap K_x)\cong \partial(L_x\cap K_x)$ by assigning to each point $y\in\partial(C_x\cap K_x)$ the unique point at which the graph of $\chi_x$ intersects the fibre over $y.$ Then $C_{xj}=\ph^{-1}(\Si^0_{xj})$ and
\e\l{obst7}
\int_{\Si^0_{xj}}\Ps=\int_{C_{xj}}\ph^*\Ps.
\e
Finally, by the definition of $Z(L_x)\in H^{m-1}(C_x)\cong H^{m-1}(C_x\cap S^{2m-1})$ the right-hand side of \eq{obst7} may be identified with the paring $\sum_{C_{xj}\sb Y}[C_{xj}\cap S^{2m-1}]\cdot Z(L_x).$ This with \eq{obst2}--\eq{obst7} implies (ii) of Condition \ref{cond}.
\qed

\section{First Approximate Solutions}\l{app sec}
Let the hypotheses of Theorem \ref{1} hold. We work in the circumstances of Definitions \ref{dfn: Wein X} and \ref{dfn: L_x}. We begin by introducing the cylinder metric on the cone $C_x\sb\C^m$ for $x\in X^\sing.$
\begin{dfn}
Denote by $r:C_x\to(0,\iy)$ the Euclidean distance from $0\in\C^m,$ which we call the {\it radius function} on $C_x.$ The {\it cylinder metric} on $C_x$ is then the conformal change by $r^{-2}$ of the cone metric on $C_x.$ This is the product metric $(\d\log r)^2+g_x$ where $g_x$ is the induced metric on $C_x\cap S^{2m-1}$ (on which the cone and cylinder metrics induce the same metric).  

We say that a smooth function $T_x:C_x\cap U_x\to\R$ is {\it of order $\de>0$} if for $k=0,1,2,\dots$ we have $|\nb^k T_x|=O(r^\de)$ where $|\,\,|,\nb$ are computed pointwise on $C_x\cap U_x$ with respect to its cylinder metric. This is equivalent to saying that $|\nb^k T_x|=O(r^{\de-k})$ with respect to the cone metric on $C_x.$
\end{dfn}

Fix $(s,t)\in\cG.$ Recall from \cite[Theorem 2.19(b)]{J1} that there exists on $X$ a $1$-form $\al$ satisfying the following condition.
\begin{cond}\l{cond: al}\iz\item[\bf(i)]
There exists $\de>0$ independent of $s,t$ and such that for each $x\in X^\sing$ there exists on $X\cap U_x\cong C_x\cap U_x$ a $0$-form $T_x$ of order $\de>0$ and with $\al=(1+\be_x)Y(L_x)+\d T_x.$
\item[\bf(ii)]
Give $(M;\om,J)$ the K\"ahler metric and $X\sb M$ the induced metric; then $\d^*\al=0.$
\item[\bf(iii)]
$\d\al$ is a compactly supported $2$-form on $X,$ whose compact-support cohomology class is the image under the natural map $\bop_{x\in X^\sing} H^1(C_x)\to H^2_c(X)$ of the element $(1+\be_x)Y(L_x)\in \bop_{x\in X^\sing} H^1(C_x).$ 
\iz
\end{cond}
We prove a lemma we will use shortly.
\begin{lem}
For $(s,t)\in\cG$ the graph of $t^2\al$ in the Weinstein neighbourhood of $X\sb (M,\om^s)$ is a Lagrangian submanifold.
\end{lem}
\begin{proof}
Let $\nu^s$ be as in Definition \ref{dfn: Wein X}. Since the condition (i) of Theorem \ref{obst thm} holds for $(s,t)\in\cG$ (as supposed in Theorem \ref{1}) it follows that the compact-support cohomology class $[\nu^s]$ is the image under the natural map $\bop_{x\in X^\sing} H^1(C_x)\to H^2_c(X)$ of the element $t_x^2Y(L_x)\in \bop_{x\in X^\sing} H^1(C_x).$ The latter is represented by $t^2\d\al,$ by Condition \ref{cond: al} (iii). Since the natural projection $V\to H^2_c(X)$ is an isomorphism it follows therefore that $\nu^s=t^2\d\al.$ But the symplectic form in the Weinstein neighbourhood of $X\sb (M,\om^s)$ is equal to $\wh\om+\nu^s,$ which thus vanishes on the graph of $t^2\al.$
\end{proof}

We define a compact connected Lagrangian submanifold $N\sb (M,\om^s).$ 
\begin{dfn}\l{dfn: N}
We define $N$ by gluing together $\bigsqcup_{x\in X^\sing}(L_x\cap K_x),$ $\bigsqcup_{x\in X^\sing}(C_x\cap U_x\-t_xK_x)$ and $X\-\bigsqcup_{x\in X^\sing} U_x.$ For $x\in X^\sing$ put $t_x:=\sqrt{1+\be_x}t.$ Notice that $t_x=O(t)$ because $\be_x$ decays to $0.$ Identify each $L_x\cap K_x$ with $t_x(L_x\cap K_x)\sb U_x$ and its image under $\Up^s_x:U_x\to M.$ This is then a Lagrangian submanifold of $(M,\om^s).$ 

We interpolate between $t_x(L_x\cap K_x)\sb U_x$ and the graph of $t^2\al$ in the Weinstein neighbourhood of $X\cap\bigsqcup_{x\in X^\sing}U_x.$ Recall from Hypothesis \ref{hp} that each $L_x$ approaches $C_x$ with order $\le0$ at infinity.  This implies that for $x\in X^\sing$ and $k=0,1,2,\dots$ we have $|\nb^k\chi_x|=O(r^{-1})$ on $L_x\-K_x\cong C_x\-K_x$ with respect to the cylinder metric. Identify $Y(L_x)\in H^1(C_x)$ with the harmonic $1$-form on the link $C_x\cap S^{2m-1}$ and identify this further with its pull-back under the projection $C_x\to C_x\cap S^{2m-1}.$ We thus regard $Y(L_x)$ as a closed $1$-form on $C_x.$ Recall then from \cite[Proposition 7.6]{J1} that the difference $\chi_x-Y(L_x)$ is an exact $1$-form on $C_x\-K_x.$ Write this as $\d E_x$ with $E_x$ a $C^\iy$ function $C_x\-K_x\to\R$ decaying at infinity to $0.$ By \cite[Theorem 7.11(b)]{J1} there is in fact a constant $\la<0$ such that for $k=0,1,2,\dots$ we have $|\nb^k\chi_x|=O(r^{\la})$ with respect to the cylinder metric on $C_x\-K_x.$ As $X^\sing$ is finite, we can suppose that $\la$ is independent of $x$ after we make it smaller if we need. Thus $L_x\-K_x$ is the graph of $Y(L_x)+\d E_x,$ and $t_x(L_x\-K_x)$ the graph of $t_x^2Y(L_x)+t_x^2\d t_{x*}E_x.$ 

On the other hand, recall from Definition \ref{dfn: ics} that $X$ approaches with order $>2$ at each $x\in X^\sing$ the cone $C_x.$ So there exists $\mu>2$ such that for $k=0,1,2,\dots$ we have $|\nb^k\et_x|=O(r^{\mu-1})$ where $|\,\,|,\nb$ are computed pointwise on $X\cap U_x\cong C_x\cap U_x$ with respect to its cylinder metric. As $X^\sing$ is finite, we can suppose that $\mu$ is independent of $x$ after we make it smaller if we need. It is easy to show (as in \cite[Lemma 4.5]{J1}) that $\et_x$ is an exact $1$-form $\d A_x$ such that for $k=0,1,2,\dots$ we have $|\nb^kA_x|=O(r^\mu)$ with respect to the cylinder metric on $C_x\cap U_x.$ Condition \ref{cond: al}(i) and Definition \ref{dfn: Wein X} imply now that the graph of $t^2\al$ in the Weinstein neighbourhood of $X\cap U_x$ corresponds to the graph of $t^2(1+\be_x)Y(L_x)+t^2\d T_x+\d A_x$ over $C_x\cap U_x.$

Let $\ta\in(\frac23,1)\cap(\frac2\mu,1)$ be a constant independent of $s,t.$ For $x\in X^\sing$ and $\ep>0$ denote by $B_x(\ep)\sb U_x$ the ball of Euclidean distance $<\ep$ from $0\in\C^m.$ Let $f_x:U_x\to[0,1]$ be a $C^\iy$ function with $f_x=0$ on $B_x(t^\ta)$ and with $f_x=1$ outside $B_x(2t^\ta).$ Consider the graph of
\e
t_x^2\d[(1-f_x)t_{x*}E_x]+t_x^2Y(L_x)+t^2\d(f_x T_x)+\d(f_x A_x) 
\e  
on $C_x\cap U_x\- K_x.$ This smoothly joins to $t_x(L_x\cap K_x)\sb U_x$ and the graph of $t^2\al$ in the Weinstein neighbourhood of $X\cap U_x.$ As a result we get in $(M,\om^s)$ a compact Lagrangian submanifold 
\e\l{N}
N:=\bigsqcup_{x\in X^\sing}(L_x\cap K_x)\sqcup \bigsqcup_{x\in X^\sing}(C_x\cap U_x\- t_xK_x)\sqcup \Bigl(X\-\bigsqcup_{x\in X^\sing} U_x\Bigr).\e
The hypothesis in the first paragraph of Theorem \ref{1} implies that $N$ is connected.
\end{dfn}

We define radius functions $r:X\to(0,\iy),$ $r:\bigsqcup_{x\in X^\sing}:L_x\to(0,\iy)$ and $r:N\to (0,\iy).$ We use the same symbol $r$ for short although these are of course different functions. 
\begin{dfn}
Let $r:X\to(0,\iy)$ be a smooth function which is nearly constant on $X\-\bigsqcup_{x\in X^\sing}U_x$ and such that for each $x\in X^\sing$ the restriction of $r:X\to(0,\iy)$ to $X\cap U_x$ is equal to the restriction of $r:C_x\to(0,\iy)$ to $C_x\cap U_x$ under the diffeomorphism $X\cap U_x\cong C_x\cap U_x$ (corresponding to the $1$-form $\et_x$).

For $x\in X^\sing$ let $r :L_x\to(0,\iy)$ be a smooth function which is nearly constant on $K_x$ and agrees outside $K_x$ with the radius function on $C_x\-K_x$ under the diffeomorphism $L_x\cong K_x\cong C_x\-K_x$ corresponding to the $1$-form $\chi_x.$ Note that $r:\bigsqcup_{x\in X^\sing} C_x\to(0,\iy),$ $r:X\to(0,\iy)$ and $r:\bigsqcup_{x\in X^\sing}:L_x\to(0,\iy)$ are all independent of $s,t.$

There is now a unique smooth function $r:N\to(0,\iy),$ depending smoothly on $s,t$ and such that the following holds: for $x\in X^\sing$ the restriction of $r:N\to(0,\iy)$ to the first piece $t_x(L_x\cap K_x)$ in \eq{N} is equal to the restriction of $t_xr:L_x\to(0,\iy)$ to the same region $t_x(L_x\cap K_x)\sb t_xL_x;$ for $x\in X^\sing$ the restriction of $r:N\to(0,\iy)$ to the second piece $C_x\cap U_x\-t_xK_x$ in \eq{N} is equal to the restriction of $r:C_x\to(0,\iy)$ to the same region $C_x\cap U_x\-t_xK_x\sb C_x;$ and the restriction of $r:N\to(0,\iy)$ to the last piece $(X\-\bigsqcup_{x\in X^\sing} U_x)$ in \eq{N} is equal to the restriction of $r:X\to(0,\iy)$ to the same region $(X\-\bigsqcup_{x\in X^\sing} U_x)\sb X.$

Note that $r:N\to (0,\iy)$ is bounded below by $t$ up to a constant independent of $s,t.$
\end{dfn}

We define conical and cylindrical metrics on $X,\bigsqcup_{x\in X^\sing}L_x$ and $N.$ 
\begin{dfn}\l{dfn: metrics}
For $x\in X^\sing$ the {\it conical metric} on $L_x$ is the Riemannian metric induced from the Euclidean metric on $\C^m.$ The {\it cylindrical metric} on $L_x$ is its conformal change by $r^{-2}:L_x\to(0,\iy).$ The {\it conical metric} on $X$ is the Riemannian metric induced from the K\"ahler metric of $(M;\om,J).$  The {\it cylindrical metric} on $X$ is its conformal change by $r^{-2}:X\to(0,\iy).$ For each $s,t$ the {\it conical metric} on $N$ is the Riemannian metric induced from the K\"ahler metric of $(M;\om^s,J^s).$ The {\it cylindrical metric} on $N$ is its conformal change by $r^{-2}:N\to(0,\iy).$
\end{dfn}

We use H\"older norms with exponent $\frac12\in(0,1);$ the choice of the exponent is unimportant and any other constant in $(0,1)$ will do. Put
\e\l{ka} \ka:=\min\{4-4\ta,(\mu-2)\ta, (2-\la)(1-\ta),2+\ta(\de-2)\}>0.\e
Since $\ta\in(\frac2\mu,1),\mu>2,\la<0$ and $\de>0$ it follows that $\ka>2-2\ta.$ We state now the key estimates for $\Im\Om^s|_N.$
\begin{prop}\l{near prop}
There exists $c>0$ independent of $s,t$ and such that
\begin{align}
\l{near1}
&\ts |r^{-m}\Im\Om|_N|_{C^{1/2}}\le cr\text{ on $\bigsqcup_{x\in X^\sing}(N\cap B_x(t^\ta))$},\\
\l{near2} 
&\ts |r^{-m}\Im\Om|_N|_{C^{1/2}}\le c(|s|r+ t^\ka)
\text{ on $\bigsqcup_{x\in X^\sing}(B_x(2t^\ta)\-B_x(t^\ta))$ and}\\
\l{near3}
&\ts |r^{-m}\Im\Om|_N|_{C^{1/2}}\le c(|s|r+ t^4r^{-4})\text{ on $N\-\bigsqcup_{x\in X^\sing}B_x(2t^\ta)$ }
\end{align}
where the H\"older norms are computed with respect to the cylindrical metric of $N.$
\end{prop}

\begin{rmk}\l{worse rmk}
Proposition \ref{near prop} with $s=0$ is an analogue of \cite[Theorem 6.7]{J4} (our $\ka$ corresponding to Joyce's $\ka-1$). He estimates $\Im\Om|_N$ in the $L^{\frac{2m}{m+2}}$ norm with respect to the conical metric. The $L^{\frac{2m}{m+2}}$ estimate is crucial to \cite[Theorem 5.3]{J3}; see also Remark \ref{exp}. The computation after \cite[(57)]{J4} shows that the $L^{\frac{2m}{m+2}}$ is good enough only for $m<6,$ which is why the dimension hypothesis is made in \cite{J3,J4}.

These do not matter to us because we do not use \cite[Theorem 5.3]{J3} to produce the second approximate solutions. Also we estimate $\Im\Om|_N$ with respect to the cylindrical metric, and in this sense the estimates of Proposition \ref{near prop} are different from those of \cite[Theorem 6.7]{J4}. 
\end{rmk}

\section{Proof of \eq{near1}--\eq{near3}}\l{Proof of 1--3}

We use the following lemma which we can prove by straightforward computation:
\begin{lem}\l{near lem}
Let $(\Ps^s)_{s\in\R^n}$ be a $C^\iy$ family of $p$-forms on $U_x,$ with $\Ps^s|_0=0.$ Then there exist $\ep,c>0$ independent of $s$ and such that the following holds: let $\xi$ be a closed $1$-form on an open set in $C_x\cap U_x$ with $|\xi|_{C^{3/2}}<\ep$ with respect to the cylinder metric on $C_x;$ then the graph of $\xi$ lies in $U_x,$ and for $s\in\R^n$ with $|s|<\ep$ we have
\e\l{near lem1} |r^{-p}\Ps^s|_{\gr\xi}|_{C^{1/2}}\le c r\text{ at every point of }\gr\xi\e
where $|\,\,|_{C^{1/2}}$ is computed with respect to the cylinder metric under the diffeomorphism $\gr\xi\cong C_x.$ \qed
\end{lem}
\begin{rmk}
Here $|\xi|_{C^{3/2}}$ is taken into account because $\Ps^s|_{\gr\xi}$ is over a fixed point of $C'$ a smooth function of $\xi,\nb\xi$ whose $C^{\be}$ norms are bounded by $|\xi|_{C^{3/2}}.$
\end{rmk}
\begin{cor}\l{near cor0}
Let $(\Ps^s)_{s\in\R^n}$ be as in Lemma \ref{near lem} and suppose moreover that $\Ps^0=0$ at every point of $U_x.$ Then there exist again $\ep,c>0$ such that the same statment holds with the upper bound $c|s|r$ in place of $cr$ in \eq{near lem1}.
\end{cor}
\begin{proof}
Putting $s=(s_1,\dots,s_n)\in\R^n$ and using Taylor's formula, write $\Ps^s=s_1\ps^s_1+\dots+s_n\ps^n_1.$ We can then apply Lemma \ref{near lem} to $\ps^s_1,\dots,\ps^s_n$ and the rest is easy.
\end{proof}

We prove now the estimates \eq{near1}--\eq{near3}.
\begin{proof}[Proof of \eq{near1}]
Fix $x\in X^\sing.$ Notice that there is on $K_x\sb\C^m$ a $C^\iy$ family $(\Xi^{ts})_{t\in[-1,1],s\in\R^n}$ of $m$-forms defined by saying that for $t\in[-1,1]\-\{0\}$ we have $\Xi^{ts}:=\Im (t_x^{-m}t_x^*\Om^s-\ps^s(x)^m\Om')$ and for $t=0$ we have $\Xi^{0s}:=0.$ This is then smooth with respect to $t$ and there exists a constant $c>0$ independent of $s,t$ and such that for $s,t$ near $(0,0)$ we have
\e\l{n0} |r^{-m}\Xi^{ts}|_{L_x}|_{C^{1/2}}\le c|t|\text{ at every point of }L_x\cap K_x\e
On the other hand, since $L_x\sb(\C^m;\om',J',\Om')$ is a special Lagrangian submanifold it follows that $\Xi^{ts}|_{L_x}=\Im(t_x^{-m}t_x^*\Om^s)|_{L_x}.$ Simple computation shows that $r^{-m}\Im(t_x^{-m}t_x^*\Om^s)|_{L_x}=t_x^*(r^{-m}\Im\Om^s|_{t_xL_x}).$ On the other hand, $t_xL_x=N$ in $t_xK_x.$ Combining these with \eq{n0} we see that
\e\l{n1} |t_x^*(r^{-m}\Im\Om^s|_N)|_{C^{1/2}}\le c|t|\text{ at every point of }t_x^{-1}N\cap K_x.\e 
This rescaled by $t_x$ implies $|r^{-m}\Im\Om^s|_N|_{C^{1/2}}\le c|t|$ at every point of $N\cap t_xK_x.$ But the radius function $r$ is, at every point of $N\cap t_xK_x,$ bounded below by $t$ up to constant. Making $c$ larger if we need, therefore, we have
\e\l{n2} |r^{-m}\Im\Om^s|_N|_{C^{1/2}}\le cr\text{ at every point of }N\cap t_xK_x.\e
We prove that the same estimate holds also on $N\cap B_x(t^\ta)\-t_xK_x.$ Applying Lemma \ref{near lem} to $\Ps^s:=\Im(\Om^s-\ps^s(x)^m\Om')$ we see that making $c$ larger if we need, we have
\e\l{n3} |r^{-m}\Ps^s|_N|_{C^{1/2}}\le cr\text{ at every point of }N\cap B_x(t^\ta)\-t_xK_x.\e
Since $N\cap B_x(t^\ta)=t_xL_x \cap B_x(t^\ta)\sb(\C^m;\om',J',\Om')$ is a special Lagrangian submanifold it follows that $\Ps^s|_N=\Im\Om^s|_N.$ The estimate \eq{n3} implies therefore that $|r^{-m}\Im\Om^s|_N|_{C^{1/2}}\le cr$ at every point of $N\cap B_x(t^\ta)\-t_xK_x,$ completing the proof.
\end{proof} 

\begin{proof}[Proof of \eq{near2}]
If $\xi$ is a $C^1$ small $1$-form on $X$ then we can write $\Im\Om|_{\gr\xi}=:P_X(\xi,\nb\xi){\rm dvol}$ where $P_X$ is a $\xi$-independent smooth function $T^*X\oplus (T^*X)^{\otimes 2}\to\R,$ and dvol is the volume form of the conical metric on $X.$ Recall from \cite[Proposition 6.3]{J2} that
\e\l{n5}
P_X(\xi,\nb\xi)=-\d^*(\ps^m \xi)+O(r^{-2}|\xi|^2)+O(|\nb\xi|^2)
\e
where $|\,\,|,\nb$ and the $O$ terms are computed with respect to the conical metric on $X.$ If we replace the conical metric by the cylindrical metric then $|\xi|$ becomes $r^{-1}|\xi|,$ $|\nb\xi|$ becomes $r^{-2}|\nb\xi|$ and the formula becomes
\e
P_X(\xi,\nb\xi)=-\d^*(\ps^m \xi)+r^{-4}O(|\xi|^2)+r^{-4}O(|\nb\xi|^2)
\e
where $|\,\,|,\nb$ and the $O$ terms are computed now with respect to the cylindrical metric on $X.$ The same result \cite[Proposition 6.3]{J2} computes also the derivatives of \eq{n5} and in particular we can estimate the H\"older norm of $P_X(\xi,\nb\xi).$ If we use again the cylindrical metric then we have
\e\l{n6}
|P_X(\xi,\nb\xi)|_{C^{1/2}}=|\d^*(\ps^m \xi)|_{C^{1/2}}+r^{-4}O(|\xi|^2)+r^{-4}O(|\nb\xi|^2);
\e
that is, these $|\,\,|,\nb$ and $O$ terms are computed with respect to the cylindrical metric on $X.$

Fix $x\in X^\sing.$ The estimate \eq{near2} is concerned with the region $B_x(2t^\ta)\- B_x(t^\ta)$ in which $N\cap B_x(2t^\ta)\- B_x(t^\ta)$ is the graph of a $1$-form $\ze$ on $C_x\cap B_x(2t^\ta)\- B_x(t^\ta)$ such that $\ze-t^2\al$ is an exact $1$-form $\d F.$ The $0$-form $F$ is made from $A_x, t_x^2t_{x*}E_x$ and $t^2 T_x$ using cut-off functions. 

We compute the right-hand side of \eq{n6}. For $\xi=t^2\al$ the $\d^*$ term vanishes and the last two quadratic terms on \eq{n6} contribute $r^{-4}t^4.$ But $r\ge t^\ta$ so this is bounded above by $t^{4-4\tau}\le t^\ka.$ For $\xi=\d A_x,t_x^2t_{x*}\d E_x$ and $t^2\d T_x$ the dominant term on the right-hand side of \eq{n6} is the $\d^*$ term. Since $\d^*$ is computed with respect to the conical metric it follows that there exists $c>0$ independent of $s,t$ and such that
\e |\d^*(\ps^m \xi)|_{C^{1/2}}\le c r^{-2}|\xi|_{C^{3/2}}\text{ at every point of }C_x\cap B_x(2t^\ta)\- B_x(t^\ta)\e
where the H\"older norms are computed with respect to the cylinder metric.  Notice that
\begin{align*}
&\xi= \d A_x=O(r^\mu)\Rightarrow
 r^{-2}|\xi|_{C^{3/2}}=r^{\mu-2}\le 2^{\mu-2} t^{\ta(\mu-2)}\le 2^{\mu-2}t^\ka,\\
&\xi=t_x^2t_{x*}\d E_x=t^{2-\la}O(r^\la)\Rightarrow
r^{-2}|\xi|_{C^{3/2}}=t^{2-\la}r^{\la-2}\le t^{(2-\la)(1-\tau)}\le t^\ka,\\
&\xi=t^2\d T_x=t^2O(r^\de)\Rightarrow
r^{-2}|\xi|_{C^{3/2}}=t^2r^{\de-2}\le t^{2+\tau(\de-2)}\le t^\ka.
\end{align*}
It is easy to show that the cut-off functions used to patch together $A_x, t_x^2t_{x*}E_x$ and $t^2 T_x$ have only negligible contributions, so the $1$-form $\d F_x$ contributes $t^\ka$ up to constant. We have thus 
\e\l{n7}|P_X(\ze,\nb\ze)|_{C^{1/2}}\le c t^\ka\text{ at every point of }B_x(2t^\ta)\-B_x(t^\ta)\e
after we make $c$ larger if we need. 

Since $r^{-m}\Im\Om|_N=P_X(\ze,\nb\ze)r^{-m}{\rm dvol}$ and since $r^{-m}{\rm dvol}$ is the volume form of the cylinder metric it follows that 
\e\l{n8} |r^{-m}\Im\Om|_N|_{C^{1/2}}=|P_X(\ze,\nb\ze)r^{-m}{\rm dvol}|_{C^{1/2}}=|P_X(\ze,\nb\ze)|_{C^{1/2}}.\e
Applying Corollary \ref{near cor0} to $\Im(\Om^s-\Om)$ we get 
\e\l{n9} |r^{-m}\Im(\Om^s-\Om)|_N|_{C^{1/2}}\le  |s|r\text{ at every point of }B_x(2t^\ta)\-B_x(t^\ta).\e
Combining \eq{n7}--\eq{n9} we get \eq{near2}.
\end{proof}

\begin{proof}[Proof of \eq{near3}]
The same computation as in \eq{n8} shows that
\e\l{n10} \ts |r^{-m}\Im\Om|_N|_{C^{1/2}}=|P_X(t^2\al,t^2\nb\al)|\text{ at every point of }N\- \bigsqcup_{x\in X^\sing}B_x(2t^\ta).\e
Applying \eq{n6} to $\xi=t^2\al$ we get
\e\l{n11}
|P_X(t^2\al,t^2\nb\al)|=r^{-4}O(t^4|\al|^2)+r^{-4}O(t^4|\nb\al|^2)=t^4O(r^{-4}).
\e
On the other hand, we show now that 
\e\l{n12}
\ts|r^{-m}\Im(\Om^s-\Om)|_N|_{C^{1/2}}\le  |s|r\text{ at every point of }N\-\bigsqcup_{x\in X^\sing} B_x(2t^\ta).
\e
As the proof of \eq{n9} is valid also on $N\cap \bigsqcup_{x\in X^\sing}(U_x\- B_x(2t^\ta)),$ the estimate \eq{n12} does hold on this region. Since $N\- \bigsqcup_{x\in X^\sing}U_x$ is compact and $t$-independent it follows that
\e\l{n13} 
\ts|r^{-m}\Im(\Om^s-\Om^0)|_N|_{C^{1/2}}\le  |s|\text{ at every point of }N\- \bigsqcup_{x\in X^\sing}U_x.
\e
But $r:X\to(0,\iy)$ is bounded below on this region, so \eq{n13} implies that \eq{n12} holds on the same region. Thus \eq{n12} holds everywhere on $N\-\bigsqcup_{x\in X^\sing} B_x(2t^\ta).$ This combined with \eq{n10} and \eq{n11} implies \eq{near3}.
\end{proof}

\section{Quadratic Estimates}\l{pert2}
Since $\ka>2-2\ta$ and $\ta\in(\frac23,1)$ it follows that
\e\l{nu'} \nu':=\max\left\{\frac{3\ta-2}{\ta-2},\frac{2-2\ta-\ka}{2-\ta},\frac{2\ta-2}{2-\ta}\right\}<0.\e 
Fix $\nu\in(\nu',0).$ Making $\nu$ larger if we need, suppose also that $2-2\nu<\th$ where $\th$ is the constant $>2$ given in Theorem \ref{1}. Give $C^{5/2}(N,\R)$ a weighted H\"older norm $|\bu|_\nu$ with weight $r^{-\nu};$ that is, for $u\in C^{5/2}(N,\R)$ define $|u|_\nu:=|r^{-\nu}u|_{C^{5/2}}$ where the right-hand side is computed with respect to the cylindrical metric on $N.$ In the same way, give $C^{1/2}(N,\R)$ the weighted H\"older norm $|\bu|_{\nu-2}$ with weight $r^{-\nu+2}.$ We prove then a corollary of Proposition \ref{near prop}.
\begin{cor}\l{near cor}
There exists $c>0$ independent of $s,t$ and such that
\e|r^{-n} \Im\Om|_N|_{\nu-2}\le c|s|+t^\th\text{ at every point of $N.$}\e
\end{cor}
\begin{proof}
From \eq{near1}--\eq{near3} we get
\begin{align}
\l{near4}
&|r^{-n} \Im\Om|_N|_{\nu-2}\le cr^{2-\nu}r=r^{3-\nu}\le c t^{\ta(3-\nu)},\\
\l{near5}
&|r^{-n} \Im\Om|_N|_{\nu-2}\le c r^{2-\nu}(|s| t^\ta+t^\ka )\le c |s|+ct^{\ta(2-\nu)+\ka}\text{ and }\\
\l{near6}
&|r^{-n} \Im\Om|_N|_{\nu-2}\le cr^{2-\nu}(|s| r+t^4r^{-4})\le c|s|+ct^{4-\ta(2+\nu)}
\end{align}
respectively. As $\nu\in(\nu',0)$ it follows from \eq{nu'} that the respective powers of $t$ in the last terms of \eq{near4}--\eq{near6} are all $>2-2\nu.$
\end{proof}

We define now a neighbourhood of $N\sb(M,\om^s)$ by gluing together the respective Weinstein neighbourhoods of 
$\bigsqcup_{x\in X^\sing}t(L_x\cap K_x),$ $\bigsqcup_{x\in X^\sing}(C_x\cap U_x\-t_xK_x)$ and $X\-\bigsqcup_{x\in X^\sing} U_x$ in \eq{N}. We call this the {\it Weinstein} neighbourhood of $N\sb(M,\om^s).$ Using it we define the graph of a $C^1$ small $1$-form on $N.$ If $\xi$ is such a $1$-form then we define $P\xi\in C^0(N,\R)$ by $(P\xi){\rm dvol}=\Im\Om^s|_{\gr \xi},$ using the diffeomorphism $\gr\xi\cong N,$ where dvol is the volume form of the conical metric on $N.$ If $\xi$ is moreover $C^{3/2}$ then $P\xi\in C^{1/2}(N,\R).$

Thus $P$ is an operator from an open subset of $C^{3/2}(T^*N)$ to the whole space $C^{1/2}(N,\R).$ This is a smooth map and its linearization at $0\in C^{3/2}(T^*N)$ is well defined. Denote the latter by $\De$ and put $Q:=P-P0-\De.$ We prove then a quadratic estimate for $Q.$ Give $C^{3/2}(T^*N)$ the weighted H\"older norm $|\bu|_\nu$ with weight $r^{-\nu}.$ 
Notice that if $|\xi|_\nu=O(t^{\th+\nu})$ then $|\xi|_{C^{3/2}}=O(t^\th)$ and so the $1$-form $\xi$ and the rescaled $1$-form $t^{-2}\xi$ are both small enough for their graphs to be well defined. 

\begin{prop}\l{quad prop}
For $c_0>0$ independent of $s,t$ there exists $c_1>0$ independent of $s,t$ and such that for $\ph,\ph'\in C^{3/2}(T^*N)$ with $|\ph|_\nu+|\ph'|_\nu\le c_0 t^{\th+\nu}$ we have
\e\l{quad}   |Q \ph'-Q \ph|_{C^{1/2}}\le  c_1r^{-4}|\ph'-\ph|_{C^{3/2}}(|\ph|_{C^{3/2}}+|\ph'|_{C^{3/2}})\text{ at every point of }N\e
where $r$ is the radius function on $N.$
\end{prop}
\begin{proof}
Put $V:=T^*N\op T^*N^{\ot2},$ a vector bundle over $N.$ Re-write $P\ph$ as $P(\ph,\nb \ph)$ where the latter $P$ is a $\ph$-independent smooth function $V\to\R.$ Denote then by $\partial P:V^*\ot V\to\R$ the partial derivative of $P$ in the fibres of $V$ (for which we do not have to choose a connection on $V$). Define also the second derivative $\partial^2P:V^*\otimes V^*\otimes V\to\R.$ For $a\in\R$ put $\ph^a:=\ph+a(\ph'-\ph).$ Put also $\Ph^a:=(\ph^a,\nb \ph^a),$ which is a section of $V.$ Put $\Ph=\Ph^0$ and $\Ph':=\Ph^1.$ Taylor's theorem implies then
\e\l{quad1} Q \ph'-Q \ph= (\Ph'-\Ph)\mathbin{\lrcorner}\int_0^1\int_0^1\Ph^a\mathbin{\lrcorner}\partial^2 P(ba \Ph^a) \d a \d b.\e
Taking the $C^{1/2}$ norms we get
\e\l{quad2}|Q \ph'-Q \ph|_{C^{1/2}}\le |\Ph'-\Ph|_{C^{1/2}}\int_0^1\int_0^1|\Ph^a|_{C^{1/2}} |\partial^2 P(ba \Ph^a)|_{C^{1/2}} \d a \d b.\e
On the other hand,
\e |\Ph'-\Ph|_{C^{1/2}}\le |\ph'-\ph|_{C^{1/2}}+|\nb\ph'-\nb\ph|_{C^{1/2}}\le |\ph'-\ph|_{C^{3/2}}\e
and we can compute $|\Ph^a|_{C^{1/2}}$ in the same way. In a way similar to proving \eq{n6} we can show that $|\partial^2P(ba\Ph^a)|_{C^{1/2}}= O(r^{-4}).$ So \eq{quad2} implies \eq{quad}.
\end{proof}

Passing to the weighted norms we get
\begin{cor}\l{cor: quad}
For $c_0>0$ independent of $s,t$ there exists $c_1>0$ independent of $s,t$ and such that for $\ph,\ph'\in C^{3/2}(T^*N)$ with $|\ph|_\nu+|\ph'|_\nu\le c_0 t^{\th+\nu}$ we have
\e   |Q \ph'-Q \ph|_{\nu-2}\le c_1 t^{\nu-2} |\ph'-\ph|_\nu (|\ph|_\nu+|\ph'|_\nu).\e
\end{cor}
\begin{proof}
After we pass to the weighted norms, the factor $r^{-4}$ in \eq{quad} becomes $r^{-4}r^\nu r^\nu r^{-\nu+2}=r^{\nu-2}.$ But the radius function $r:N\to(0,\iy)$ is bounded below by $t$ up to positive constant, so we have the factor $t^{\nu-2}$ as above.
\end{proof}

\section{Uniform Invertibility}\l{pert3}
For each connected component $Y\sb X$ choose a $C^\iy$ function $E_Y:X\to\R$ supported in $Y\-\bigsqcup_{x\in X^\sing}U_x$ and with $\int_Y E_Y{\rm dvol}= 1$ where dvol is the volume form of the conical metric on $Y.$ Regard this as an $\R$-linear map $\R\to C^{1/2}(N,\R)$ which maps $1$ to $E_Y.$ Varying $Y\in\pi_0X$ we get a basis of the $\R$-vector space $H^0(X,\R)$ and the family $(E_Y)_{Y\in \pi_0X}$ defines then an $\R$-linear map $H^0(X,\R)\to C^{1/2}(N,\R)$ which we denote by $E.$ Consider now the operator
\e\l{unif0} \De\op E: C^{5/2}(N,\R)\op H^0(X,\R)\to C^{1/2}(N,\R).\e
The $E$ factor compensates for the one-dimensional cokernel of $\De: C^{5/2}(N,\R)\to C^{1/2}(N,\R)$ so that $\De\oplus E$ is surjective. We prove
\begin{prop}\l{unif prop}
There exists an $(s,t)$ uniformly bounded right inverse to \eq{unif0} with respect to the weighted norms on the H\"older spaces.
\end{prop}

The idea of the proof is that we can decompose the operator \eq{unif0} into two pieces. One depends only on $(L_x)_{x\in X^\sing}$ and the other only on $X.$ Both are uniformly invertible, from which we deduce that the original operator is invertible. The heart of the proof is therefore to make a reasonable definition of the decomposition. This is rather long but more or less straightforward. We do it following \cite[\S7.2]{DK} in outline, although the background geometry is different. We begin by introducing the cut-off functions with which we will glue together the operators we have defined.
\begin{dfn}\l{dfn: cut-off}
Let $F_X:N\to[0,1]$ and $F_L:N\to[0,1]$ be a partition of unity on $N$ subordinate to $N\-\bigsqcup_{x\in X^\sing} B_x(t^\ta)$ and $N\cap \bigsqcup_{x\in X^\sing} B_x(2t^\ta),$ $F_X$ supported on the former and $F_L$ on the latter. For each connected component $Y\sb X$ put $F_Y:=F_X|_Y.$

Choose a constant $\rho<\ta$ independent of $s,t.$ Put $L:= \bigsqcup_{x\in X^\sing}L_x$ and let $G_L:L\to[0,1]$ be a cut-off function supported on $L\cap\bigsqcup_{x\in X^\sing} B_x(t^\rho),$ with $G_L=1$ on $L_x\cap \bigsqcup_{x\in X^\sing} B_x(2t^\ta)$ and satisfying the following condition: there exists $c>0$ independent of $s,t$ and such that at every point of $L$ we have $|G_L|_{C^{5/2}}\le c(-\log t)^{-1}$ where the H\"older norm is computed with respect to the cylindrical metric on $L.$ 

We do something similar on $X.$ Choose a constant $\si>\ta$ independent of $s,t.$ Let $G_X:X\to[0,1]$ be a cut-off function supported on $X\-\bigsqcup_{x\in X^\sing} B_x(t^\si),$ with $G_X=1$ on $X\-\bigsqcup_{x\in X^\sing} B_x(t^\ta)$ and satisfying the following condition: there exists $c>0$ independent of $s,t$ and such that at every point of $X$ we have $|G_X|_{C^{5/2}}\le c(-\log t)^{-1}$ where the H\"older norm is computed with respect to the cylindrical metric on $X.$ For each connected component $Y\sb X$ put $G_Y:=G_X|_Y.$
\end{dfn}

We define two operators $\De_L,D_L:C^{5/2}_\nu(L,\R)\to C^{1/2}_{\nu-2}(L,\R).$
\begin{dfn}\l{dfn: D1}
Denote by $g'$ the Euclidean metric of $\C^m,$ and by the same $g'$ the induced metric on $L=\bigsqcup_{x\in X^\sing}L_x\sb \bigsqcup_{x\in X^\sing}\C^m.$ Define $\De_L:C^{5/2}_\nu(L,\R)\to C^{1/2}_{\nu-2}(L,\R)$ by $\De_L:=-\d^{*\prime}\d$ where $\d^{*\prime}$ is the formal adjoint of $\d$ computed with respect to $g'.$ This is a bounded operator with respect to the weighted H\"older norms. It is invertible, by \cite[Theorem 6.2.15]{Marsh}. 

We choose on $L$ a Riemannian metric $g^{st}$ depending smoothly on $s,t$ and such that the restriction to $L_x\-B_x(t_x^{-1}t^\rho)$ of $g^{st}$ agrees with $g',$ and the restriction to each $L_x\cap B_x(2t_x^{-1}t^\ta)$ of $g^{st}$ agrees with the metric induced by the embedding $L_x\cap B_x(2t_x^{-1}t^\ta)\cong t_xL_x\cap B_x(2t^\ta)\sb U_x$ where $U_x\sb M$ is given the K\"ahler metric corresponding to $(\om^s,J^s).$ Choose a smooth function $\ps^{st}:L\to(0,\iy)$ depending smoothly on $s,t,$ with $\ps^{st}=1$ on each $L_x\-B_x(t_x^{-1}t^\rho),$ and with $\ps^{st}=t_x^*\ps^s$ on each $L_x\cap B_x(2t_x^{-1}t^\ta).$ Define $D_L:C^{5/2}_\nu(L,\R)\to C^{1/2}_{\nu-2}(L,\R)$ by $D_L:=-\d^{*st}[(\ps^{st})^m\d]$ where $\d^{*st}$ is the formal adjoint of $\d$ computed with respect to $g^{st}.$
\end{dfn}
We show that $D_L$ is uniformly invertible.
\begin{lem}
$D_L:C^{5/2}_\nu(L,\R)\to C^{1/2}_{\nu-2}(L,\R)$ is invertible and its inverse is bounded uniformly with respect $s,t.$
\end{lem}
\begin{proof}
Computation similar to the proof of \eq{near1} shows that there exists $c>0$ independent of $s,t$ and such that 
\e\l{gg} |g^{st}-g'|_{C^{3/2}}\le ctr\le 2ct^\rho\text{ at every point of }L\e
where the $C^{3/2}$ norm is computed with respect to the cylindrical metric on $L.$
Denote by $\nb^{st},\nb'$ the respective Levi-Civita connections of $g^{st},g'.$ The estimate \eq{gg} implies then that there exists  $c>0$ independent of $s,t$ and such that for $\ph\in C^{3/2}(T^*L)$ we have
\e |\nb^{st}\ph-\nb'\ph|_{C^{1/2}}\le ct^\rho|\ph|_{C^{3/2}}\text{ at every point of }L.\e
In general, the operator $\d^*$ of a Riemannian metric $g_{ab}$ is of the form $-g^{ab}\nb_b$ in the index notation. In the current circumstances, as we work with the cylindrical metrics, the $g^{ab}$ term contributes $r^{-2}$ so that
\e\l{r-2} |\d^{*st}\ph-\d^{*\prime}\ph|_{C^{1/2}}\le cr^{-2}t^\rho|\ph|_{C^{3/2}}\text{ at every point of }L.\e
Hence making $c$ larger if we need, we find that for $u\in C^{5/2}(L,\R)$ we have
\e   |\d^{*st}[(\ps^{st})^m\d u]-\d^{*\prime}[(\ps^{st})^m\d u]|_{C^{1/2}}\le cr^{-2}t^\rho|u|_{C^{5/2}}.\e
Since $\ps^{00}=1$ it follows that making $c$ larger if we need, we have
\e   |D_Lu-\De_L u|_{C^{1/2}}\le cr^{-2}(|s|+t^\rho) |u|_{C^{5/2}}.\e
Passing to the weighted norms we get
\e\l{DD}   |D_L u-\De_L u|_{\nu-2}\le c(|s|+t^\rho) |u|_\nu.\e
Recall now from \cite[Theorem 6.2.15]{Marsh} that the operator $\De_L:C^{5/2}_\nu(L,\R)\to C^{1/2}_{\nu-2}(L,\R)$ is invertible. Since $g^{st}$ and $\ps^{st}$ depend smoothly on $s,t$ it follows that the operator $D_L:C^{5/2}_\nu(L,\R)\to C^{1/2}_{\nu-2}(L,\R)$ is bounded uniformly with respect to $s,t$ (where the operator norm is computed with respect to the weighted norms). The estimate \eq{DD} implies therefore that that the operator $D_L\De_L^{-1}-\id:C^{1/2}_{\nu-2}(L,\R)\to C^{1/2}_{\nu-2}(L,\R)$ has operator norm $\le c(|s|+t^\rho)\ll1$ with respect to the $C^{1/2}_{\nu-2}$ norm. Thus $D_L\De_L^{-1}$ is invertible and its inverse is bounded uniformly with respect to $s,t.$ On the other hand, $\De_L$ has nothing to do with $s,t$ and is therefore bounded uniformly with respect to $s,t.$ Accordingly so is $\De_L^{-1}(D_L\De_L^{-1})^{-1},$ which is a right inverse to $D_L.$ Changing the order of $D_L$ and $\De_L^{-1}$ we see also that $\De_L^{-1}D_L$ is invertible and that $D_L$ has a left inverse. The latter must agree with the right inverse, which completes the proof.
\end{proof}

We define two other operators $\De_Y\oplus E_Y$ and $D_Y\oplus E_Y.$
\begin{dfn}\l{dfn: D2}
Recall that $N\-\bigsqcup_{x\in X^\sing} t_xK_x$ lies in the Weinstein neighbourhood of $X\sb(M,\om^s)$ and write this as the graph of a $1$-form $\xi$ on $X\-\bigsqcup_{x\in X^\sing} t_xK_x.$ Introduce on $X$ a $1$-form $\xi''$ which agrees with $\xi$ on $X\-\bigsqcup_{x\in X^\sing} B_x(t^\ta)$ and vanishes on $\bigsqcup_{x\in X^\sing} B_x(2t^\si).$ Fix a connected component $Y\sb X.$ Define $\De_Y:C^{5/2}_\nu(Y,\R)\to C^{1/2}_{\nu-2}(Y,\R)$ by $\De_Y=-\d^*[(\ps^0)^m\d]$ where $\d^*$ is computed with respect to the conical metric on $Y\sb X.$ Since $E_Y:Y\to\R$ is compactly supported it follows that this is an element of the weighted space $C^{1/2}_{\nu-2}(N,\R)$ and defines therefore an $\R$-linear map $\R\to C^{1/2}_{\nu-2}(N,\R)$ which maps $1$ to $E_Y.$ We denote the latter also by $E_Y.$ Combining it with $\De_Y$ we get an operator
\e\l{DE} \De_Y\oplus E_Y: C^{5/2}_\nu(Y,\R)\oplus \R\to C^{1/2}_{\nu-2}(Y,\R).\e
Denote by $g''$ the Riemannian metric on $X$ induced by the embedding $X\cong\gr\xi''\sb M$ where $M$ is given the K\"ahler metric corresponding to $(\om^s,J^s).$ Define $D_Y:C^{5/2}_\nu(Y,\R)\to C^{1/2}_{\nu-2}(Y,\R)$ by $D_Y:=-\d^{*\prime\prime}[(\ps^s)^m\d]$ where $\d^{*\prime\prime}$ is computed with respect to $g''.$ 
\end{dfn}
We show that \eq{DE} has a uniformly bounded right inverse.
\begin{lem}\l{lem: right inverse}
Using the $C^{5/2}_\nu$ norm on $C^{5/2}_\nu(Y,\R)$ and the Euclidean norm on $\R$ give $C^{5/2}_\nu(Y,\R)\oplus \R$ the product norm. The operator \eq{DE} has then a right inverse which is uniformly bounded with respect to $s,t.$
\end{lem}
\begin{proof}
Denote by $g$ the conical metric on $Y\sb X.$  We show that there exists $c>0$ independent of $s,t$ and for which
\e\l{gg'} |g''-g|_{C^{3/2}}\le ct^{2-2\si}\text{ at every point of }X\e
where the $C^{3/2}$ norm is computed with respect to the cylindrical metric on $X.$ It is clear that this holds on $\bigsqcup_{x\in X^\sing} B_x(t^\si)$ because $g''=g$ there. On $X\-\bigsqcup_{x\in X^\sing}B_x(2t^\ta)$ we have $\xi=t^2\al$ which implies the stronger estimate $|g''-g|_{C^{3/2}}=O(t^2).$ On each $B_x(2t^\ta)\-B_x(2t^\si)$ the $1$-form $\xi$ is made from $\d A_x,t_x^2t_{x*}\d E_x,\d T_x$ and $t^2\al.$ Computation similar to the proof of \eq{near2} shows that the dominant contribution is that of $t^2\al,$ which is $r^{-2}O(t^2),$ including the rescale factor $r^{-2}.$ This with $r\ge t^\si$ implies \eq{gg'}. 

If we denote by $\nb'',\nb$ the respective Levi-Civita connections of $g'',g$ then we find from \eq{gg'} a constant $c>0$ independent of $s,t$ and such that for $\ph\in C^{3/2}(T^*Y)$ we have
\e |\nb''\ph-\nb \ph|_{C^{1/2}}\le ct^{2-2\si}|\ph|_{C^{3/2}}\text{ at every point of }Y.\e
Computing the $\d^*$ operators as in \eq{r-2} we get
\e |\d^{*\prime\prime}\ph-\d^*\ph|_{C^{1/2}}\le cr^{-2}t^{2-2\si}|\ph|_{C^{3/2}}\text{ at every point of }Y.\e
Hence making $c$ larger if we need, we find that for $u\in C^{5/2}(Y,\R)$ we have
\e   |\d^{*\prime\prime}[(\ps^s)^m\d u]-\d^*[(\ps^s)^m\d u]|_{C^{1/2}}\le cr^{-2}t^{2-2\si}|u|_{C^{5/2}}.\e
Since $\ps^0=\ps$ it follows that making $c$ larger if we need, we have
\e   |D_Yu-\De_Yu|_{C^{1/2}}\le cr^{-2}(|s|+t^{2-2\si}) |u|_{C^{5/2}}.\e
Passing to the weighted norms we get
$|D_Yu-\De_Yu|_{\nu-2}\le c(|s|+t^{2-2\si}) |u|_\nu;$ that is,
\e\l{DDe}
D_Y-\De_Y:C^{5/2}_\nu(Y,\R)\to C^{1/2}_{\nu-2}(Y,\R) \text{ has operator norm }\le c(|s|+t^{2-2\si})
\e
with respect to the weighted H\"older norms. 

Recall now from \cite[Theorems 2.14 and 2.16(b)]{J1} that the Sobolev space version of \eq{DE} is surjective, and from \cite[Theorem 6.9]{Marsh} that the H\"older version holds too; that is, \eq{DE} itself is surjective. So there exists a bounded linear operator $\Th:  C^{1/2}_{\nu-2}(Y,\R)\to C^{5/2}_\nu(Y,\R)\oplus \R$ which is a right inverse to \eq{DE}. Hence it follows by \eq{DDe} that the operator $(D_Y\oplus E_Y)\Th-\id:C^{1/2}_{\nu-2}(Y,\R)\to C^{1/2}_{\nu-2}(Y,\R)$ has operator norm less than $c(|s|+t^{2-2\si})\ll1.$ So $(D_Y\oplus E_Y)\Th$ is invertible and its inverse is bounded uniformly. As $\Th$ has nothing to do with $s,t$ it is also bounded uniformly with respect to $s,t.$ Accordingly so is $\Th[(D_Y\oplus E_Y)\Th]^{-1},$ which is a right inverse to $D_Y\oplus E_Y.$ This completes the proof.
\end{proof}

We prove Proposition \ref{unif prop} now.
\begin{proof}[Proof of Proposition \ref{unif prop}]
For each connected component $Y\sb X$ denote by $R_Y\oplus \pi_Y: C^{1/2}_{\nu-2}(Y,\R)\to C^{5/2}_\nu(Y,\R)\oplus \R$ the uniformly bounded right inverse produced in Lemma \ref{lem: right inverse}. Define $S:C^{1/2}(N,\R)\to C^{5/2}(N,\R)$ by $S:=G_LD_L^{-1} F_L+\sum_{Y\in\pi_0X}G_YR_YF_Y.$ Define $\varpi: C^{1/2}(N,\R)\to H^0(X,\R)$ by $v\mapsto\sum_{Y\in\pi_0X}\pi_Y F_Yv.$ Recall from Definitions \ref{dfn: D1} and \ref{dfn: D2} that $D_L=\De$ on the support of $G_L$ and that for each connected component $Y\sb X$ we have $D_Y=\De$ on the support on $G_Y.$ So for $v\in C^{1/2}(N,\R)$ we have
\ea\l{D10} 
&D_L G_LD_L^{-1}F_Lv+\sum_{Y\in \pi_0X}(D_Y G_YR_YF_Yv)+\sum_{Y\in \pi_0X} G_YE_Y\pi_Y F_Yv\\
&=\De G_LR_LF_Lv+\sum_{Y\in \pi_0X}(\De G_YR_YF_Yv)+\sum_{Y\in \pi_0X} E_Y\pi_Y F_Yv\\
&=\De S v+E\varpi v=(\De\op E)(S\op\varpi)v.
\ea
On the other hand, since $G_LF_L=F_L$ it follows that
\e\l{D11}
D_L G_LD_L^{-1}F_Lv-[D_L,G_L]F_Lv=G_LF_Lv=F_Lv.
\e
For $Y\in \pi_0X,$ since $(D_Y\oplus E_Y)(R_Y\oplus\pi_Y)=\id$ and $G_YF_Y=F_Y$ it follows that
\e\l{D12}
\sum_{Y\in \pi_0X}G_Y D_YR_YF_Yv+\sum_{Y\in \pi_0X}G_YE_Y\pi_Y F_Yv=G_YF_Yv=\sum_{Y\in \pi_0X}F_Yv.
\e
Since $F_L$ and $(F_Y)_{Y\in\pi_0X}$ form a partition of unity it follows that the rightmost term $F_Lv$ on \eq{D11} and the rightmost term
$\sum_{Y\in \pi_0X}F_Yv$ on \eq{D12} sum up to $v.$ This with \eq{D10}--\eq{D12} implies
\e
(\De\op E)(S\op\varpi)v-v=[D_L,G_L]R_LF_Lv+\sum_{Y\in\pi_0X}[D_Y,G_Y]R_YF_Yv.
\e
Since the bracket terms contains derivatives of $G_L,G_Y,$ which are $O((-\log t)^{-1}),$ we get a constant $c>0$ independent of $s,t,v$ and satisfying
\e| (\De\op E)(S\op\varpi)v-v|_{C^{1/2}}\le c(-\log t)^{-1} |v|_{C^{1/2}}\e
where the $C^{1/2}$ norms are computed with respect to the cylindrical metric on $N.$ Passing to the weighted norms we get
\e| (\De\op E)(S\op\varpi)v-v|_{\nu-2}\le c(-\log t)^{-1} |v|_{\nu-2}.\e
The operator $(\De\op E)(S\op\varpi)-\id:C^{1/2}(N,\R)\to C^{1/2}(N,\R)$ has thus operator norm $\le c(-\log t)^{-1}\ll1.$ So $J:=(\De\op E)(S\op \varpi)$ is invertible and $J^{-1}$ is bounded uniformly with respect to $s,t.$ Now $J^{-1}(S\op\varpi)$ is a right inverse to $\De\op E$ which is bounded uniformly with respect to $s,t,$ which completes the proof of Proposition \ref{unif prop}.
\end{proof}

\section{Second Approximate Solutions}\l{pert4}
We solve the special Lagrangian equation with an additional term. Recall from \S\ref{pert2} that $\th>2-2\nu.$ 
\begin{cor}\l{sol cor}
There exists $c>0$ independent of $s,t$ and such that the following holds: for $(s,t)\in\cG$ with $|s|<t^\th$ there exists $u\op w\in C^{5/2}_\nu(N,\R)\op H^0(X,\R)$ such that $P\d u+E w=0$ with $|u|_\nu\le c t^{\th+\nu}.$
\end{cor}
\begin{proof}
Denote by $R\op \pi: C^{1/2}_{\nu-2}(N,\R)\to  C^{5/2}_\nu(N,\R)\op H^0(X,\R)$ the uniformly bounded right inverse to $\De\op E$ produced in Proposition \ref{unif prop}. Put $B:=\{v\in C^{1/2}_{\nu-2}(N,\R):|v|_{\nu-2}\le t^{\th+\nu}\}.$ Since $R$ is uniformly bounded it follows that there exists $c_0>0$ independent of $s,t$ and such that for $v\in B$ we have $|Rv|_\nu\le c_0t^{\th+\nu},$ which implies that we can define $Q\d Rv$ as in Proposition \ref{quad prop}. We can therefore define $T:B\to C^{1/2}_{\nu-2}(N,\R)$ by $v\mapsto-P0-Q\d Rv.$ 

We show that $T$ maps $B$ to itself. Recall from Corollary \ref{cor: quad} that there exists $c_1>0$ independent of $s,t$ and such that for $v\in B$ we have
\e\l{s1} |Q\d Rv|_{\nu-2}\le  c_1t^{\nu-2}|Rv|_{\nu-2}^2\le c_1c_0^2t^{\nu-2}t^{2(\th+\nu)},\e
where the last inequality follows from the definition of $B.$ On the other hand, by Corollary \ref{near cor} we have
\e\l{s2} |P0|_{\nu-2}\le |s|+t^\th \le 2t^\th.\e
This and \eq{s1} imply that 
\e\l{s3} |Tv|_{\nu-2} \le |P0|_{\nu-2}+|Q\d Rv|_{\nu-2}\le(2t^{-\nu}+c_1c_0^2t^{\th+2\nu-2})t^{\th+\nu}.\e
Since $\th>2-2\nu$ and $\nu>0$ it follows by \eq{s3} that for $t$ small enough we have $Tv\in B.$

We have thus defined the map $T:B\to B.$ For $v,v'\in B$ we have
\ea &|Tv-Tv'|_{\nu-2}  \le c_1 t^{\nu-2} |Rv-Rv'|_\nu (|Rv|_\nu+|Rv'|_\nu)\\
&\le c_1c_0^2 t^{\nu-2} |v-v'|_{\nu-2}(|v|_{\nu-2}+|v'|_{\nu-2})\le c_1c_0^2 t^{\th+2\nu-2}|v-v'|_{\nu-2}.\ea
Hence we see, using again $\th+2\nu-2>0,$ that for $t$ small enough the map $T:B\to B$ is a contraction map. As $B$ is closed we can apply to it the fixed point theorem; that is, there exists $v\in B$ with $Tv=v.$ 

Since $R\oplus \pi$ is a right inverse to $\De\oplus E$ it follows that putting $u\op w:=(R\op \pi)v$ we have $(\De\op E)(u\op w)=v.$ Now
\e
P\d u+Ew=P0+ Qu+\De u+Ew=P0+Q\d Rv+ v=-Tv+v=0.\e
Moreover $|u|_\nu\le c_0|v|_{\nu-2}\le c_0t^{\th+\nu}$ as we have to prove.
\end{proof}

We prove now Theorem \ref{1} for $X$ connected.
\begin{proof}[Proof of Theorem \ref{1} for $X$ connected]
Regard $w\in H^0(X,\R)\cong\R$ as a real constant. Since $E_X:N\to[0,1]$ is supported in $N\-\bigsqcup_{x\in X^\sing}t_xK_x\cong X\-\bigsqcup_{x\in X^\sing}t_xK_x$ it follows that
\e\l{add0}
w\int_X E_X\,{\rm dvol}=w\int_N E_X\,{\rm dvol}=-\int_NP\d u\,{\rm dvol}=-\int_N\Im\Om^s.
\e
Here the right-hand side is the pairing of $[N]\in H_m(M,\Z)$ and $[\Im\Om^s]\in H^m(M,\R).$ Definition \ref{dfn: N} shows that $[N]$ converges, as $(s,t)$ tends to $(0,0),$ to $[X]\in H_m(N,\Z).$ But the integer homology classes are discrete, so $[N^{st}]=[X]$ for $(s,t)$ close enough to $(0,0).$ The right-hand side of \eq{add0} is therefore equal to the pairing $[X]\cdot[\Im\Om^s].$ In the current case, as $X$ is connected the left-hand side of \eq{eq for Om} is the sum over all connected components of $C_{xj}\sb C_x.$ But the pairing $\sum_{C_{xj}\sb C_x} [C_{xj}\cap S^{2m-1}]\cdot Z(L_x)$ is equal to the integral of $\Im\Om$ over a compact set in $L_x,$ which vanishes automatically. So the left-hand side of \eq{eq for Om} vanishes, and accordingly $[X]\cdot[\Im\Om^s]=0.$ Thus $w\int_X E_X\,{\rm dvol}=0.$ But $E_X$ is so chosen that $\int_X E_X\,{\rm dvol}\ne0$ and we have therefore $w=0.$ This means that the graph of $\d u$ is a special Lagrangian submanifold of $(M;\om^s,J^s,\Om^s).$
\end{proof}

When $X$ is not connected, we deal with the additional term $w$ as follows.
\begin{lem}\l{add lem}
Suppose that for each $(s,t)\in\cG$ we are given $u\op w\in C^{5/2}_\nu(N,\R)\op H^0(X,\R)$ with $P\d u+Ew=0.$ Then there exist $c,\ep>0$ independent of $s,t$ and such that if we write $w=\bop_{Y\in\pi_0X}w_Y\in H^0(X,\R)$ then 
$\max_{Y\in \pi_0X}|w_Y|\le c t^{m+\ep}.$
\end{lem}
\begin{proof}
We proceed as in the proof of Theorem \ref{obst thm}. For each connected component $Y\sb  X$ there is a unique connected component of $(\gr\d u)\-(\bigsqcup_{x\in X^\sing}t_xK_x)$ that is $C^1$ close to $Y\-\bigsqcup_{x\in X^\sing} t_xK_x,$ which we denote by $Y^u.$ This $Y^u$ is a manifold with boundary, with $\partial Y^u\sb \bigsqcup_{x\in X^\sing} t_xK_x.$ For $x\in X^\sing$ let $\Si_{xj}\sb t_x^{-1}(Y^u\cap t_xK_x)$ be a connected component, which is a compact $m-1$ manifold diffeomorphic to some connected component $C_{xj}\cap S^{2m-1}\sb C_x\cap S^{2m-1}.$ Denote by $0*\Si_{xj}\sb \C^m$ the union of the line segments between the origin $0$ and the points of $\Si_{xj}.$ The union $Y^u\cup \bigsqcup_{C_{xj}\sb Y} t_x(0*\Si_{xj})$ defines in $M$ an $m$-cycle homologous to $\bar Y.$ Integrating $\Im\Om^s$ over this we get
\e\l{add2} [\bar Y]\cdot [\Im\Om^s]=\sum_{C_{xj}\sb Y}\int_{t_x(0*\Si_{xj})}\Im\Om^s+\int_{Y^u}\Im\Om^s.\e
Following with modification the proof of Theorem \ref{obst thm} we find $\ep>0$ such that
\e\l{add3}
\sum_{C_{xj}\sb Y}\int_{t_x(0*\Si_{xj})}\Im\Om^s
=\sum_{C_{xj}\sb Y}t_x^m\ps^s(x)^m[C_{xj}\cap S^{2m-1}]\cdot Z(L_x)+O(t^{m+\ep}).
\e
But it is supposed in Theorem \ref{1} that $(s,t)\in\cG$ satisfies (ii) of Theorem \ref{obst thm}; that is,
\e\l{add4}
[\bar Y]\cdot [\Im\Om^s]=\sum_{x\in C_{xj}}t^m\ps^s(x)^m(1+\ga_x)[C_{xj}\cap S^{2m-1}]\cdot Z(L_x).
\e
Recalling that $t_x=\sqrt{1+\be_x}t$ and that $\be_x,\ga_x$ decay with power orders it follows from \eq{add2}--\eq{add4} that
\e\l{add5}
\int_{Y^u}\Im\Om^s=O(t^{m+\ep}).
\e
Notice that $\Im\Om^s|_{\gr \d u}=(P\d u){\rm dvol}=-(Ew){\rm dvol}.$ By the definition of $E$ we have $Ew=\sum_{Y\in\pi_0X}w_YE_Y.$ Since each $E_Y$ is supported in $Y^u\cong Y\-\bigsqcup_{x\in X^\sing}t_xK_x$ it follows that
\e\l{add6}
\int_{Y^u}\Im\Om^s=-\int_{Y^u}(Ew){\rm dvol}=-w_Y\int_Y E_Y{\rm dvol}=-w_Y.
\e
It follows from \eq{add5} and \eq{add6} that $w_Y=O(t^{m+\ep}).$
\end{proof}

\section{Proof of Theorem \ref{1}}\l{Proof of 1}
We begin by recalling from \cite{J3} the key result we will use.
\begin{thm}[Theorem 5.3 of \cite{J3}]\l{5.3}
For every $m\in\{3,4,5,\dots\},c>0$ and $\ep>0$ there exists $t_0>0$ such that if $t\in(0,t_0)$ then the following holds. Let $(M;\om,J,\Om)$ be an almost Calabi--Yau $m$-fold. Define $\ps:M\to(0,\iy)$ by \eq{ps}. Let $N\sb (M,\om)$ be a compact oriented Lagrangian submanifold with $\int_{N}\Im\Om=0.$ Let $g$ be the Riemannian metric on $N$ induced from the K\"ahler metric of $(M;\om,J).$ Suppose that 
\iz\item[\bf(i)]
the Riemannian manifold $(N,g)$ has injectivity radius $\ge ct$ and sectional curvature $\le ct^{-2}.$
\iz
Let $U\sb (M,\om)$ be a Weinstein neighbourhood of $N$ isomorphic to a neighbourhood of the zero-section $N\sb T^*N$ whose fibres are balls of radius $ct$ around the zero-section. Denote by $\pi:T^*N\to N$ the vector bundle projection. Notice that the Levi-Civita connection of $(N,g)$ induces over $T^*N$ a vector bundle isomorphism $T(T^*N)\cong \pi^*TN\oplus \pi^*T^*N.$ Since the two vector bundles $\pi^*TN$ and $\pi^*T^*N$ have metrics and Levi-Civita connections induced from $g,$ we get on $T(T^*N)$ a metric and a connection. Denote these by $\hat g$ and $\wh\nb.$ Suppose that 
\iz\item[\bf(ii)]
for $k=0,1,2,3$ we have $\sup_U|\wh\nb^k\Im\Om|\le ct^{-k}$ where $|\,\,|$ is computed pointwise with respect to $\hat g.$
\iz
For $p\in[1,\iy]$ denote by $L^p(N,\R)$ the $L^p$ space of the Riemannian manifold $N.$ For $p\in[1,\iy)$ and $k=0,1,2,\dots$ denote by $L^p_k(N,\R)$ the Sobolev space of the Riemannian manifold $N.$ So $L^p_0(N,\R)=L^p(N,\R).$ Suppose now that $W\sb L^2(N,\R)$ is a finite-dimenisonal $\R$-vector subspace, and $\pi_W:L^2(N,\R)\to W$ the $L^2$ orthogonal projection. Suppose that
\iz
\item[\bf(iii)] for $v\in L^2_1(N,\R)$ with $\pi_W v=0$ we have $\|v\|_{L^{\frac{2m}{m-2}}}\le c\|\d v\|_{L^2},$
\item[\bf(iv)] for $w\in W$ we have $\|\d^*\d w\|_{L^{\frac{2m}{m+2}}}\le \frac12 c\|\d w\|_{L^2}$ and
\item[\bf(v)]
for $w\in W$ with $\int_{N} w\,{\rm dvol}=0$ we have $\|w\|_{C^0}\le ct^{1-\frac m2}\|\d w\|_{L^2}.$
\iz
Denote by {\rm dvol} the volume form of the Riemannian manifold $N$ and define a smooth function $e^{i\Th}:N\to S^1$ by $\Om|_{N}=(\ps|_{N})^m e^{i\Th}{\rm dvol}.$ Denote by $\cos\Th:N\to\R$ its real part and by $\sin\Th:N\to\R$ its imaginary part. Suppose that
\iz
\item[\bf(vi)]
$\|\pi_W(\ps^m\sin\Th)\|_{L^1}\le ct^{\ep+m}$ and
\item[\bf(vii)]
for each $p\in[1,\iy]$ we have $\|\ps^m\sin\Th\|_{L^p}+t\|\d(\ps^m\sin\Th)\|_{L^p}\le  ct^{\ep+\frac{ m}{p}}.$
\iz
Then there exists on $N$ an exact $1$-form $\xi$ whose graph lies in the Weinstein neighbourhood $U\sb(M,\om)$ and defines a special Lagrangian submanifold of $(M;\om,J,\Om).$ There is moreover a constant $c'>0$ depending only on $m,c,\ep$ such that $\sup_{N}|\xi|\le c't^{1+\ep}$ where $|\,\,|$ is computed pointwise with respect to $g.$
\end{thm}
\begin{rmk}\l{exp}
Theorem \ref{5.3} is in fact slightly weaker than \cite[Theorem 5.3]{J3}. The hypothesis (vii) is stronger than the corresponding one in \cite[Theorem 5.3]{J3}. The latter requires the estimate in (vii) to hold only for $p\ge \frac{2m}{m+2}.$ This is crucial to the results of \cite{J3,J4} because in those circumstances the $L^p$ estimate is sensitive to the choice of $p;$ see also Remark \ref{worse rmk}. On the other hand, our second approximate solutions satisfy so good estimates that the $L^p$ estimate will hold for every $p\ge1.$

The exponent $\frac{2m}{m+2}$ is the dual to the exponent $\frac{2m}{m-2}$ in the uniform Sobolev inequality in (iii) above. The idea of using it goes back to \cite{Lee}.
\end{rmk}

Let $s,t$ satisfy $|s|<t^\th$ and $(s,t)\in\cG.$ Let $u\op w$ be as in Corollary \ref{sol cor} and embed $N$ into $M$ using the graphical diffeomorphism $N\cong\gr\d u.$ We show then that this embedding $N\to M$ satisfies the hypotheses (i)--(vii) of Theorem \ref{5.3} with $(M;\om^s,J^s,\Om^s)$ in place of $(M;\om,J,\Om).$ The proof of (i)--(vi) are nearly the same as in \cite{J3,J4}. In the circumstances of \cite{J3,J4} there are only one family of approximate solutions which we perturb to the true solutions. These approximate solution satisfies (i)--(vi) of Theorem \ref{5.3}. The details about (i),(ii) are given in \cite[\S6.3]{J3} in the simplest case (amongst the results of \cite{J3,J4}). The details about (iii)--(v) are given in \cite[\S7.3]{J3} in the simplest case. The details about (vi) are given in \cite[\S7.2]{J3} in the simplest case.

Our current circumstances are more complex because the second approximate solutions are further perturbations of the first approximate solutions. But the second approximate solutions are in $\bigsqcup_{x\in X^\sing}U_x$ as close to $t_xL_x$ as the first approximate solutions are, and outside $\bigsqcup_{x\in X^\sing}t_xK_x$ as close to $X$ as the first approximate solutions are. These properties are all we need for (i)--(v) of Theorem \ref{5.3}. For (vi) we use also Lemma \ref{add lem}.

The proof of (vii) is the difference from \cite{J3,J4} as we show now. Since $u\op w$ is as in Corollary \ref{sol cor} it follows that $\Im\Om|_{N}=-\sum_{Y\in\pi_0X}w_YE_Y.$ But as in Lemma \ref{add lem} every constant $w_Y$ is $O(t^{m+\ep}).$ Since each $E_Y$ is supported in $Y\-\bigsqcup_{x\in X^\sing} U_x$ it follows that for every $p\in[1,\iy]$ we have
\e\l{9last} \|\ps^m\sin\Th\|_{L^p}+\|\d(\ps^m\sin\Th)\|_{L^p}
\le ct^{m+\ep},\e
which implies (vii). We can thus apply Theorem \ref{5.3} to $N,$ which completes the proof of Theorem \ref{1}. \qed

\section{Parameter Differentiability}\l{sect10}
In this section we prove the general results we shall need from analysis, concerning smooth families of linear elliptic operators over compact manifolds. The notation of this section is slightly different from that of the other sections. 

For a manifold $Y$ we denote by $C^\iy(Y)$ the $\R$-vector space of $C^\iy$ functions from $Y$ to $\R.$ Let $X$ be a compact Riemannian manifold, $d\ge0$ an integer and $(\De_t:C^\iy(X)\to C^\iy(X))_{t\in\R^d}$ a family of second-order elliptic linear differential operators with $C^\iy$ coefficients. We suppose that $\De_t$ is $C^\iy$ with respect to $t$ in the following sense. For $u\in C^\iy(\R^d\times X)$ define $\De u:\R^d\times X\to \R$ by assigning to $(t,x)\in\R^d\times X$ the value $(\De_tu|_{\{t\}\times X})(x).$ This $\De u:\R^d\times X\to \R$ should then be $C^\iy.$ Equivalently, the coefficients of the differential operators are $C^\iy$ functions of the two variables $t,x.$ Let $t_1,\dots,t_d$ be the coordinates of $\R^d.$ Then $\De:C^\iy(\R^d\times X)\to C^\iy(\R^d\times X)$ is a linear differential operator which does not contain $\frac{\partial}{\partial t_1},\dots, \frac{\partial}{\partial t_d}.$

Fix $\al\in(0,1).$ For $k=0,1,2,\dots$ and $v\in C^\iy(X)$ denote by $\|v\|_{C^{k,\al}}$ the H\"older $C^{k,\al}$ norm of $v$ relative to the given Riemannian metric of $X.$ For $u\in C^\iy(\R^d\times X)$ denote by $\|u\|_{C^{k,\al}}:\R^d\to\R$ the function $t\mapsto \|u|_{\{t\}\times X}\|_{C^{2,\al}}$ and by $\|u\|_{C^0}:\R^d\to\R$ the function $t\mapsto \max_{x\in X}|u(t,x)|.$ These may therefore be thought of as the $C^{k,\al}$ and $C^0$ norms in the $X$-direction. They satisfy the following Schauder estimate.  
\begin{thm}\l{thm: Schaud}
Let $X,\De$ be as above and $\Om\sb\R^d$ a compact subset. Then there exists $A>0$ such that for every $u\in C^\iy(\R^d\times X)$ we have, at every point of $\Om,$ $\|u\|_{C^{2,\al}}\le A(\|\De u\|_{C^{\al}}+\|u\|_{C^0}).$ 
\end{thm} 
\begin{proof}
The point is that $A$ is independent of $t\in\Om.$ As in \cite[Theorem 6.2]{GT} such constants depend on the coefficients of the differential operators after writing them with local coordinates. But as $\Om$ is compact the coefficients are $t$-uniformly bounded and so there exists an upper bound for $A.$
\end{proof}
\begin{rmk}
Here the Riemannian metric of $X$ is independent of $t.$ We can also consider a smooth family of Riemannian metrics and the corresponding H\"older norms. But as in the statement above we are interested in the estimates for $t\in\Om$ with $\Om$ compact, and the Riemannian metrics (and the H\"older norms) are uniformly equivalent for $t\in \Om.$ So essentially we need only one Riemannian metric.     
\end{rmk}

Denote by ${\rm vol}$ the volume density of the Riemannian manifold $X.$ We prove a corollary of Theorem \ref{thm: Schaud}.
\begin{cor}\l{cor: Schaud}
Let $X,\De$ be as above and $\Om\sb\R^d$ a compact subset. Suppose that for each $t\in\R^d$ the vector subspace $\ker\De_t\sb C^\iy(X)$ is one-dimensional and spanned by constant functions. Then there exists $A>0$ with the following property: let $u\in C^\iy(\R^d\times X)$ be such that for $t\in\Om$ we have $\int_X u|_{\{t\}\times X}{\rm dvol}=0;$ then at every point of $\Om$ we have $\|u\|_{C^{2,\al}}\le A\|\De u\|_{C^{\al}}.$ 
\end{cor} 
\begin{proof}
If $A$ does not exist then there exist for every $n=0,1,2,\dots$ some $u_n\in C^\iy(X)$ with $\int_X u_n{\rm dvol}=0$ and some $t_n\in\Om$ with $\|u_n\|_{C^{2,\al}}> n\|\De_{t_n}u_n\|_{C^{2,\al}}.$ Since $\Om$ is compact we find, passing to a subsequence if we need, that $t_n$ converges to some $t_\iy\in\Om.$ As $\|u_n\|_{C^{2,\al}}> n\|\De_{t_n}u_n\|_{C^{k,\al}}\ge0$  we can define $\|u_n\|^{-1}_{C^{2,\al}}.$ Put $v_n:=\|u_n\|^{-1}_{C^{2,\al}}u_n.$ Then $\|v_n\|_{C^{2,\al}}=1$ and $\|\De_{t_n}v_n\|_{C^{\al}}<\frac1n.$ Since the H\"older norms are $n$-uniformly equivalent it follows that $\|v_n\|_{C^{2,\al}}$ is $C^{2,\al}$ bounded. Passing to a subsequence if we need, we see that $v_n$ converges in the $C^{2}$ sense to some $C^2$ function $v_\iy:X\to\R.$ Recalling that $\|\De_{t_n}v_n\|_{C^{\al}}<\frac1n$ we find $\De_{t_\iy}v_\iy=0.$ So $\De_{t_\iy}v_\iy=0;$ that is, $v_\iy$ is constant. But since $\int_X u_n{\rm dvol}=0$ it follows that $\int_Xv_\iy{\rm dvol}=0$ and hence that $v_\iy=0.$ So $v_n$ converges in the $C^2$ sense to $0.$ Thus $\|\De_{t_n} v_n\|_{C^{\al}}+\|v_n\|_{C^0}$ converges to $0.$ By Theorem \ref{thm: Schaud} therefore $\|v_n\|_{C^{2,\al}}$ converges to $0,$ which however contradicts $\|v_n\|_{C^{2,\al}}=1$ and thus completes the proof.
\end{proof}

We denote by $\nb$ the Levi-Civita connection of $X.$ If $u=(u_t\in C^\iy(X))_{t\in\R^d}$ is a $C^\iy$ family then for $l=0,1,2,\dots$ define
\begin{align*}
\partial_t^l u:=
\Bigl(\Bigl(\frac{\partial}{\partial t_1}\Bigr)^{q_0}\dots\Bigl(\frac{\partial}{\partial t_d}\Bigr)^{q_d}u\Bigr)_{\substack{q_0,\dots,q_d\in\{0,\dots,l\}\\ q_0+\dots+q_d=l}}\text{ and }\\
\|\partial_t^l u\|_{C^{k,\al}}:=
\max_{\substack{q_0,\dots,q_d\in\{0,\dots,l\}\\ q_0+\dots+q_d=l}}
\Bigl\| \Bigl(\frac{\partial}{\partial t_1}\Bigr)^{q_0}\dots\Bigl(\frac{\partial}{\partial t_d}\Bigr)^{q_d}u\Bigr\|_{C^{k,\al}.} 
\end{align*}
We prove a statement to the effect that if $\De_tu_t$ is $C^\iy$ with respect to $t$ then $u_t$ is $C^\iy.$ More precisely, the following holds. We give a proof for the sake of clarity although it is similar to that of \cite[Theorem 7.5]{Kod}.
\begin{thm}\l{3}
Let $X,\De$ be as above. Suppose that for each $t\in\R^d$ the vector subspace $\ker\De_t\sb C^\iy(X)$ is one-dimensional and spanned by constant functions. Let $(u_t\in C^\iy(X))_{t\in\R^d}$ be such that for $t\in\R^d$ we have $\int_X u_t{\rm dvol}=0$ and such that the function $\De u:(t,x)\mapsto (\De_tu_t)(x)$ is a $C^\iy$ function from $\R^d\times X\to \R;$ then $u:(t,x)\mapsto u_t(x)$ is $C^\iy$ function from $\R^d\times X$ to $\R.$
\end{thm} 
\begin{proof}
For $l=0,1,2,\dots$ we say that a continuous function $f:\R^d\times X\to \R$ is $C^l$ with respect to $t$ if for every $p=0,1,2,\dots$ and $q=0,\dots,l$ the derivatives $\nb^p\partial_t^qf$ exist at every point of $\R^d\times X$ and are continuous sections of $(T^*X)^{\otimes p}.$ We prove that for $l=0,1,2,\dots$ the following holds:
\begin{equation}\l{3l}\parbox{10cm}{
let $(u_t\in C^\iy(X))_{t\in\R^d}$ be such that for $t\in\R^d$ we have $\int_X u_t{\rm dvol}=0$ and such that the function $\De u:(t,x)\mapsto (\De_tu_t)(x)$ is $C^l$ with respect to $t;$ then $u:(t,x)\mapsto u_t(x)$ is $C^l$ with respect to $t.$
}\end{equation}
We do this in three steps. 
\begin{proof}[Proof of \eq{3l} for $l=0$]
We show first that if $v,w\in C^0(\R^d,C^\al(X))$ then the functions $v+w,vw:t\mapsto v_t+w_t,v_tw_t$ are also $C^0$ functions from $\R^d$ to $C^\al(X).$ Let $s\in \R^d.$ Then $\lim_{t\to s}\|v_t\|_{C^\al}=\|v_s\|_{C^\al}$ and $\lim_{t\to s}\|w_t-w_s\|_{C^\al}=\lim_{t\to s}\|v_t-v_s\|_{C^\al}=0.$ These with the identity
\[v_tw_t-v_sw_s=v_t(w_t-w_s)+(v_t-v_s)w_s\] 
implies that $\lim_{t\to s}\|v_tw_t-v_sw_s\|_{C^\al}=0.$ So $vw\in C^0(\R^d,C^\al(X)).$ The proof for $v+w$ is simpler and we omit it. We show now that for $q=0,1,2,\dots$ the following holds: 
\begin{equation}\l{30}
\parbox{10cm}{
let $\xi_0,\dots,\xi_q$ be $t$-independent vector fields on $X$ and for $a=0,\dots,q$ put $\nb_a:=\nb_{\xi_a},$ the derivative in the direction $\xi_a;$ then $\nb_0\cdots\nb_q u_t\in C^0(\R^d,C^\al(X)).$
}\end{equation}
Let $s\in\R^d.$ Since $\De u\in C^\iy(\R^d\times X)$ it follows that $\lim_{t\to s}\|\De_tu_t-\De_su_s\|_{C^\al}=0.$ On the other hand, since $\De_t$ is $C^\iy$ with respect to $t$ it follows that $\lim_{t\to s}\|\De_su_s-\De_tu_s\|_{C^\al}=0.$ These with the identity
\[\De_tu_t-\De_tu_s=\De_tu_t-\De_su_s+\De_su_s-\De_tu_s\]
implies $\lim_{t\to s}\|\De_tu_t-\De_tu_s\|_{C^\al}=0.$ By Corollary \ref{cor: Schaud} therefore $\lim_{t\to s}\|u_t-u_s\|_{C^{2,\al}}=0.$ So \eq{30} holds for $q\le 1.$

Suppose next that $k\ge2$ be an integer and that \eq{30} holds for $q\le k-1.$ Computing the commutator $\De_t\nb_2\cdots\nb_k-\nb_2\cdots\nb_k\De_t$ and applying it to $u_t$ we see that the difference
\e\l{comm} \De_t\nb_2\cdots\nb_k u_t-\nb_2\cdots\nb_k\De_t u_t\e 
is a linear combination of the derivatives of order $\le k$ of $u_t.$ By the induction hypothesis these lie in $C^0(\R^d,C^\al(X)).$ Also the coefficients of the linear combination are $C^\iy$ with respect to $t$ and may therefore be regarded as elements of $C^0(\R^d,C^\al(X)).$ As $C^0(\R^d,C^\al(X))$ is closed under sums and products, the linear combination lies in $C^0(\R^d,C^\al(X));$ that is, \eq{comm} lies in $C^0(\R^d,C^\al(X)).$ Since $\De_tu_t$ is $C^\iy$ with respect to $t$ it follows that the second term of \eq{comm} lies in $C^0(\R^d,C^\al(X))$ and hence that $\De_t\nb_0\cdots\nb_k u_t\in C^0(\R^d,C^\al(X)).$ 

Put $v_t:=\nb_2\cdots\nb_k u_t$ so that $t\mapsto\De_tv_t$ lies in $C^0(\R^d,C^\al(X)).$ Letting $s\in\R^d$ we have $\lim_{t\to s}\|\De_tv_t-\De_sv_s\|_{C^\al}=0.$ On the other hand, since $\De_t$ is $C^\iy$ with respect to $t$ it follows that $\lim_{t\to s}\|\De_sv_s-\De_tv_s\|_{C^\al}=0.$ These with the identity
\[\De_tv_t-\De_tv_s=\De_tv_t-\De_sv_s+\De_sv_s-\De_tv_s\]
implies $\lim_{t\to s}\|\De_tv_t-\De_tv_s\|_{C^\al}=0.$ By the induction hypothesis, on the other hand, $\lim_{t\to s}\|v_t-v_s\|_{C^0}=0.$ By Theorem \ref{thm: Schaud} therefore $\lim_{t\to s}\|v_t-v_s\|_{C^{2,\al}}=0.$ So $\lim_{t\to s}\|\nb_0\nb_1(v_t-v_s)\|_{C^{\al}}=0;$ that is, $t\mapsto \nb_0\nb_1\nb_2\cdots\nb_k u_t$ lies in $C^0(\R^d,C^\al(X)).$ Thus \eq{30} holds for $q=k$ and the induction is complete.

It follows from \eq{30} that $u_t$ is $C^0$ with respect to $t.$
\end{proof}

\begin{proof}[Proof of \eq{3l} for $l=1$]
Take $d=1$ for the moment. The first paragraph is devoted to proving that $\frac{\d u_t}{\d t}|_{t=0}$ exists. By hypothesis $\De_tu_t$ is $C^1$ with respect to $t;$ and in particular, $\partial_t(\De_tu_t)=\frac{\d}{\d t}(\De_tu_t)$ is well defined. As the differential operator $\De_t$ depends smoothly on $t$ we can define $\partial_t\De_t=\frac{\d \De_t}{\d t}$ which is also a differential operator depending smoothly on $t.$ Define $\dot u\in C^\iy(X,\R)$ by
\[\De_0\dot u:=\frac{\d}{\d t}(\De_tu_t)\Bigm|_{t=0}-\frac{\d\De_t}{\d t}\Bigm|_{t=0}u_0\text{ and }\int_X\dot u\,\d{\rm vol}=0.\]
By simple computation we can verify that
\ea\l{De_tu_t0}
\De_t\Bigl(\frac{u_t-u_0}{t}-\dot{u}\Bigr)&=\frac{\De_tu_t-\De_0u_0}{t}-\frac{\d}{\d t}(\De_tu_t)\Bigm|_{t=0}\\
&+\Bigl(\frac{\d\De_t}{\d t}\Bigm|_{t=0}-\frac{\De_t-\De_0}{t}\Bigr)u_0
+(\De_0-\De_t)\dot{u}.
\ea
Since $\De_tu_t$ is $C^1$ with respect to $t$ (as defined just after the statement of Theorem \ref{3}) it follows that if $\xi$ is a $t$-independent vector field on $X$ then for $p=0,1$ we have
\e\l{De_tu_t1}
\lim_{t\to0}\sup_X\Bigl|\nb_\xi^p\Bigl(\frac{\De_tu_t-\De_0u_0}{t}-\frac{\d}{\d t}(\De_tu_t)\Bigm|_{t=0}\Bigr)\Bigr|=0.
\e 
As $\|\,\,\|_{C^\al}$ is the norm of $C^\al(X)$ it follows from \eq{De_tu_t1} that
\e\l{De_tu_t2}
\lim_{t\to0}\Bigl\|\frac{\De_tu_t-\De_0u_0}{t}-\frac{\d}{\d t}(\De_tu_t)\Bigm|_{t=0}\Bigr\|_{C^\al}=0.
\e
On the other hand, since $\De_t$ depends smoothly on $t$ it follows that
\e\l{De_tu_t3}
\lim_{t\to0}\Bigl\|\Bigl(\frac{\d\De_t}{\d t}\Bigm|_{t=0}-\frac{\De_t-\De_0}{t}\Bigr)u_0\Bigr\|_{C^\al}=\lim_{t\to0}\|(\De_0-\De_t)\dot{u}\|_{C^\al}=0.
\e 
By \eq{De_tu_t0}--\eq{De_tu_t3} we have $\lim_{t\to0}\displaystyle\Bigl\|\De_t\Bigl(\frac{u_t-u_0}{t}-\dot{u}\Bigr)\Bigr\|_{C^\al}=0.$ By Corollary \ref{cor: Schaud} therefore $\lim_{t\to0}\displaystyle\Bigl\|\frac{u_t-u_0}{t}-\dot{u}\Bigr\|_{C^{2,\al}}=0.$ So $\frac{\d u_t}{\d t}|_{t=0}$ exists. 

Now for the general $d\ge2$ the same proof shows that the partial derivatives $\frac{\partial u_t}{\partial t_1}|_{t=0},\dots,\frac{\partial u_t}{\partial t_d}|_{t=0}$ exist. Letting $s\in\R^d$ and considering $u_{s+t}$ in place of $u_t$ we see that the partial derivative exist everywhere. Now 
\e\l{32}
\De_t(\partial_tu_t)=\partial_t(\De_tu_t)-(\partial_t \De_t)u_t
\e
and the right-hand side is $C^0$ with respect to $t.$ Applying \eq{3l} with $l=0$ and with $\partial_t u_t$ in place of $u_t$ we find that $\partial_t u_t$ is $C^0$ with respect to $t.$ So $u_t$ is $C^1$ with respect to $t$ as we have to prove. 
\end{proof}

\begin{proof}[Proof of \eq{3l} for $l\ge2$]
Suppose now that $q\ge2$ is an integer and that \eq{3l} holds for $l\le q-1.$ Let $\De_tu_t$ be $C^q$ with respect to $t.$ By hypothesis then $u_t$ is $C^{q-1}.$ Now 
\e\l{33}
\De_t\partial_t^{q-1}u_t=\partial_t^{q-1}(\De_tu_t)-\sum_{p=0}^{q-2}(\partial_t^{q-1-p} \De_t)\partial_t^{p}u_t.
\e
Since $u_t$ is $C^{q-1}$ it follows that the right-hand side of \eq{33} is $C^1$ with respect to $t.$ Applying \eq{3l} with $l=1$ and with $\partial_t^{q-1}u_t$ in place of $u_t$ we find that $\partial_t^{q-1}u_t$ is $C^1$ with respect to $t.$ So $u_t$ is $C^q$ with respect to $t.$
\end{proof}
The statement \eq{3l} implies Theorem \ref{3}.
\end{proof}

We come now to the main theorem of the present section. There is over $X$ a vector bundle $W:=T^*X\oplus (T^*X)^{\otimes2}.$ For $v\in C^\iy(X)$ write $Dv:=\nb v\oplus \nb^2v$ which is a section of $W.$ For each $t\in\R^d$ suppose given a smooth map $F_t:W\to W^*\otimes W^*$ which commutes with the bundle projections $W\to X$ and $W^*\ot W^*\to X.$ So for each $x\in X,$ if we denote by $W_x\sb W$ the fibre over $x$ then $F_t$ maps $W_x$ to $W_x^*\otimes W_x^*.$ Suppose that for every $x\in X$ and $w\in W_x,$ if we regard $F_t(w)$ as a bilinear form on $W_x$ then it is a symmetric bilinear form. Suppose also that the function $F:(t,w)\mapsto F_t(w)$ is a smooth map from $\R^d\times W$ to $W^*\otimes W^*.$ If $u\in C^\iy(\R^d\times X)$ then we define $F(Du):\R^d\times X\to W^*\otimes W^*$ by $ (t,x)\mapsto F_t(Du(x)).$

\begin{thm}\l{5}
Let $\Om\sb\R^d$ be a compact set. Then there exists $\ep>0$ for which the following holds. Let $u_0=0\in C^\iy(\R^d\times X)$ and let $(u_n\in C^\iy(\R^d\times X))_{n=1}^\iy$ be such that for $n=1,2,3,\dots$ and for every $t\in\R^d$ we have
\[ \int_X u_n|_{\{t\}\times X}{\rm dvol}^0=0,\;\;\;  \|u_n|_{\{t\}\times X}-u_{n-1}|_{\{t\}\times X}\|_{C^{2,\al}}\le \ep\Bigl(\frac12\Bigr)^{n}.\]
Let $(\la_n\in C^\iy(\R^d\times X))_{n=0}^\iy$ be such that  for $n=1,2,3,\dots$ we have
\e\l{Deu_n} \De u_{n}=\la_n+F(Du_{n-1} )\bullet(Du_{n-1})^{\otimes 2}\e
as functions on $\R^d\times X.$ Then 
\iz
\item[\bf(i)] there exists $B>0$ independent of $t,n$ and such that $\|u_n\|_{C^{3,\al}}\le B$ at every point of $\Om;$
\item[\bf(ii)] for $k=0,1,2,\dots$ there exists $B>0$ depending on $k$ but independent of $t,n$ and such that $\|u_n\|_{C^{k+2,\al}}\le B$ at every point of $\Om;$
\item[\bf(iii)] for $\et\in(\frac12,1)$ and $k=0,1,2,\dots$ there exists $C>0$ independent of $t,n$ and such that $\|u_n-u_{n+1}\|_{C^{k+2,\al}}\le C\et^n$ at every point of $\Om;$
\item[\bf(iv)] for $k,l=0,1,2,\dots$ there exists $B>0$ depending on $k,l$ but independent of $t,n$ and such that $\|\partial_t^lu_n\|_{C^{k+2,\al}}\le B$ at every point of $\Om;$ and
\item[\bf(v)] for $\et\in(\frac12,1)$ and $k,l=0,1,2,\dots$ there exists $C>0$ independent of $t,n$ and such that $\|\partial_t^l(u_n-u_{n+1})\|_{C^{k+2,\al}}\le C\et^n$ at every point of $\Om.$
\iz
\end{thm}
Part (iii) and (v) will not be strictly necessary in practice. The reader in a hurry can skip their proofs and proceed to Remark \ref{Ascoli}.
\begin{proof}[\bf Proof of Theorem \ref{5}]
We deal with many functions from $\R^d$ to $\R$ which we estimate only on $\Om.$ We introduce therefore the following notation: for $a,b:\R^d\to\R$ write $a\le b$ if $a(t)\le b(t)$ for every $t\in\Om.$

Let $x\in X.$ Restricting $F$ to the vector space $W_x$ we get a smooth function $W_x\to(W_x^*)^{\ot2}$ and differentiating this we get a smooth function $W_x\to (W^*)^{\ot 3}.$ Varying $x\in X$ we get a smooth map $W\to (W^*)^{\ot 3}$ which we denote by $F'.$ Again for $x\in X$ define a smooth map $G_x:W_x\to W_x^*$ by $G_x(w):=F'(w )\bullet w^{\ot2}+2 F(w )\bullet w$ for $w\in W_x.$ Varying $x\in X$ we get a smooth map $W\to W^*$ which we denote by $G.$ 

Let $A$ be so large as in Theorem \ref{thm: Schaud} and Corollary \ref{cor: Schaud}. Since $X,\Om$ are compact and since $G$ maps $0\in W_x$ to $0\in W_x^*$ we get $\ep>0$ such that for every $f\in C^\iy(X)$ with $\|f\|_{C^{2,\al}}\le \ep$ we have $\|G(Df)\|_{C^{\al}}\le \frac12A^{-1}.$ Suppose now that $u_n,\la_n$ are as in Theorem \ref{5}. We prove the statements (i)--(v). 

\begin{proof}[\bf Proof of Theorem \ref{5}(i)]
Let $\xi$ be a vector field on $X$ which is independent of $t,n.$ Applying $\nb_\xi$ to \eq{Deu_n} we find a smooth map $F_0:W\oplus W\to \R$ such that for every $n=1,2,3,\dots$ we have
\e\l{F_0}
\nb_\xi\De u_n=F_0(Du_n\oplus Du_{n-1}) +G(Du_{n-1} )\bullet\nb_\xi Du_{n-1}.
\e
(Here the $\nb_\xi\la_n$ term is absorbed into the $F_0$ term, using the equation $\la_n=\De u_{n}-F(Du_{n-1} )\bullet(Du_{n-1})^{\otimes 2}$). Since $X,\Om$ are compact and $(Du_n)_{n=0}^\iy$ $C^\al$ bounded we get a constant $B>0$ independent of $t,n$ and such that for every $n\ge 1$ we have $\|F_0(Du_n\oplus Du_{n-1})\|_{C^\al}\le B.$ This with $\|G(Du_n)\|_{C^{\al}}\le \frac12A^{-1}$ implies that 
\ea\l{i0}
\|\nb_\xi\De u_n\|_{C^\al}&\le B+\frac12A^{-1}\|\nb_\xi D u_{n-1}\|_{C^{\al}}
\ea
On the other hand, since the commutator $\De\nb_\xi-\nb_\xi\De$ is a second-order differential operator we get a constant $C>0$ independent of $t,n$ and such that for every $t\in\Om$ and $n\ge0$ we have
\e\l{i1}
\|\De\nb_\xi u_n\|_{C^{\al}}\le \|\nb_\xi\De u_n\|_{C^{\al}}+C\|u_n\|_{C^{2,\al}}.
\e
As the commutator $ D\nb_\xi-\nb_\xi D$ is also a second-order differential operator, making $C$ larger if we need, it follows that for every $t\in\Om$ and $n\ge0$ we have
\e\l{i2}
\|\nb_\xi D u_n\|_{C^{\al}}\le \| D\nb_\xi u_n\|_{C^{\al}}+C\|u_n\|_{C^{2,\al}}.
\e
Recall now from Theorem \ref{thm: Schaud} that $\|\nb_\xi u_n\|_{C^{2,\al}}\le A(\|\De\nb_\xi u_n\|_{C^{\al}}+\|\nb_\xi u_n\|_{C^{\al}}).$ Combining this with \eq{i2} and making $C$ larger if we need, we see that for every $n\ge0$ we have
\e\l{i3}
\|\nb_\xi D u_n\|_{C^{\al}}\le A\|\De\nb_\xi u_n\|_{C^{\al}}+C\|u_n\|_{C^{2,\al}}.
\e
Combining this with \eq{i1} and making $C$ larger if we need, we see that for every $n\ge0$ we have
\e\l{i4}
\|\nb_\xi D u_n\|_{C^{\al}}\le A\|\nb_\xi\De u_n\|_{C^{\al}}+C\|u_n\|_{C^{2,\al}}.
\e
Combining this with \eq{i0} and making $C$ larger if we need, we see that for every $n\ge0$ we have
\e\l{i5}
\|\nb_\xi D u_n\|_{C^{\al}}\le AB+\frac12\|\nb_\xi D u_{n-1}\|_{C^{\al}}+C\|u_n\|_{C^{2,\al}}.
\e
Since $\|u_n\|_{C^{2,\al}}$ is bounded we find that making $B$ larger if we need, we have
\e\l{i6}
\|\nb_\xi Du_n\|_{C^{\al}}\le \frac12 \|\nb_\xi Du_{n-1}\|_{C^{\al}}+B.
\e 
Using \eq{i6} repeatedly we see that for every $n\ge 1$ we have
\[\|\nb_\xi Du_n\|_{C^{\al}}\le B(1+\dots+(\ts\frac12)^{n-1})\le 2B.\] 
So the sequence $(\nb_\xi Du_n)_{n=0}^\iy$ is $C^{\al}$ bounded. But $\xi$ has been an arbitrary vector field on $X,$ and $X$ is compact. These imply that the full derivative $\nb Du_n=\nb^3u_n$ is $C^{\al}$ bounded as we have to prove.
\end{proof}
\begin{proof}[\bf Proof of Theorem \ref{5}(ii)]
Let $k\ge0$ be an integer such that the sequence $(u_n)_{n=0}^\iy$ is $C^{k+2,\al}$ bounded. We prove then that $(u_n)_{n=0}^\iy$ is $C^{k+3,\al}$ bounded. 

Let $\xi_0,\dots,\xi_k$ be vector fields on $X$ which are independent of $t,n.$ For $a=0,\dots,k$ denote by $\nb_a:=\nb_{\xi_a}$ the covariant derivative in the direction $\xi_a.$ Introduce a finite set
\e\l{P} P:=\bigsqcup_{\mu=1}^{k+1} \{(p_1,\dots,p_\mu)\in\{0,1,2,\dots\}^{\t\mu}:p_1+\dots+p_\mu\le k\}.\e
For $n\ge1$ put $z_n:=Du_n\op Du_{n-1}.$ Setting $\xi=\xi_k$ and applying $\nb_0\cdots\nb_{k-1}$ to both sides of \eq{F_0} we get, after straightforward computation, a finite family $(F_{p_1\dots p_\mu}:W\op W\to (W^*)^{p_1}\otimes\dots\otimes (W^*)^{p_\mu})_{(p_1,\dots,p_\mu)\in P}$ of smooth maps such that
\ea\l{ii0}
\!\!\!
\nb_0\cdots\nb_k\De u_n&=\sum_{(p_1,\dots,p_\mu)\in P}F_{p_1\dots p_\mu}(z_n)\bullet (\nb^{p_1}z_n\otimes\dots\otimes\nb^{p_\mu}z_n)\\
&+G(Du_{n-1})\bullet(\nb_0\cdots\nb_k Du_{n-1}).
\ea
Since $X,\Om$ are compact and $(z_n)_{n=1}^\iy$ $C^\al$ bounded we get a constant $B>0$ independent of $t,n$ and such that for every $n\ge 0$ and $(p_1,\dots,p_\mu)\in P$ we have 
$\|F_{p_1\dots p_\mu}(z_n)\|_{C^\al}\le B.$
Since $u_n$ is by hypothesis $C^{k+2,\al}$ bounded it follows that for $(p_1,\dots,p_\mu)\in P$ or more generally for $p_1,\dots,p_\mu\le k$ we have $\|\nb^{p_1}z_n\|_{C^\al}\le B,\dots,\|\nb^{p_\mu}z_n\|_{C^\al}\le B.$ So for $n\ge 1$ and $(p_1,\dots,p_\mu)\in P$ we have
\e\l{ii1}
\|F_{p_1\dots p_\mu}(z_n) (\nb^{p_1}z_n\otimes\dots\otimes\nb^{p_\mu}z_n)\|_{C^\al}\le B,
\e
making $B$ larger if we need. Combining \eq{ii0} and \eq{ii1} we see that for $n\ge 0$ we have
\e\l{ii2}
\|\nb_0\cdots\nb_k\De u_n\|_{C^\al}\le
B+\frac12A^{-1} \|\nb_0\cdots\nb_k Du_n\|_{C^\al}.
\e
Generalizing \eq{i4} we get a constant $C>0$ independent of $t,n$ and such that for every $n\ge0$ we have
\ea\l{ii3}
\|\nb_0\cdots\nb_k D u_n\|_{C^{\al}}\le A\|\nb_0\cdots\nb_k\De u_{n}\|_{C^{\al}}+C\|u_n\|_{C^{k+2,\al}}\\
\le A\|\nb_0\cdots\nb_k\De u_{n}\|_{C^{\al}}+B,
\ea
making $B$ larger if we need. Combining \eq{ii2} and \eq{ii3}, and making $B$ larger if we need, we see that for $n\ge 1$ we have
\e\l{ii4}
\|\nb_0\cdots\nb_k D u_n\|_{C^{\al}}\le B+\frac12\|\nb_0\cdots\nb_k Du_{n-1}\|_{C^\al}.
\e
Using \eq{ii4} repeatedly we see then that $\|\nb_0\cdots\nb_k  Du_n\|_{C^\al}$ is uniformly bounded with respect to $n.$ But $\xi_0,\dots,\xi_k$ have been arbitrary vector fields on $X,$ and $X$ is compact. These imply that the full derivative $\nb^{k+1} Du_n=\nb^{k+3}u_n$ is $C^\al$ bounded as we have to prove. 
\end{proof}
\begin{rmk}
We can in fact unify the proofs of Theorem \ref{5}(i),(ii). Notice that the equation \eq{ii0} makes sense for $k=0,$ for which the equation is exactly \eq{F_0}. Also the structures of the estimates in the two proofs are the same. 

It would be a matter of taste whether to give a single proof or split it into the two. We have chosen the latter because we need much less notation for the first part, the proof of Theorem \ref{5}(i). 
\end{rmk}

\begin{proof}[\bf Proof of Theorem \ref{5}(iii)]
We prove that for $q=0,1,2,\dots$ the following holds:
\begin{equation}\l{iii0}\parbox{10cm}{
for every $\et\in(\frac12,1)$ there exists $C>0$ independent of $t,n$ and such that 
for every $n\ge0$ we have $\|u_n-u_{n+1}\|_{C^{q+2,\al}}\le C\et^n.$ 
}\end{equation}
We do this by an induction on $q.$ Since $\|u_n-u_{n+1}\|_{C^{2,\al}}\le \ep(\frac12)^n$ we know already that \eq{iii0} holds for $q=0.$ Let $k\ge0$ be any integer and let \eq{iii0} hold for $q=k.$ 

Let $\xi_0,\dots,\xi_k$ be vector fields on $X$ which are independent of $t,n.$ For $a=0,\dots,k$ put $\nb_a:=\nb_{\xi_a}$ the covariant derivative in the direction $\xi_a.$ Let $P$ be as in \eq{P}. Let $(p_1,\dots,p_\mu)\in P.$ Since $F^{k+1}_{p_1\dots p_\mu}:W\op W\to (W^*)^{p_1}\otimes\dots\otimes (W^*)^{p_\mu}$ is smooth we get a constant $C>0$ independent of $t,n$ and such that if two sections $z,z':X\to W\op W$ are $C^{\al}$ close enough to the zero-section then $\|F_{p_1\dots p_\mu}(z)-F_{p_1\dots p_\mu}(z')\|_{C^\al}\le C\|z-z'\|_{C^\al}.$
Recalling that $(z_n)_{n=0}^\iy$ is $C^\al$ bounded and making $C$ larger if we need, we see that for every $n\ge 1$ we have 
\e\l{iii1}
\|F_{p_1\dots p_\mu}(z_n)-F_{p_1\dots p_\mu}(z_{n+1})\|_{C^\al}\le C\|z_n-z_{n+1}\|_{C^\al}.
\e
By hypothesis there exists $C>0$ independent of $t,n$ and such that for every $n\ge0$ we have $\|u_n-u_{n+1}\|_{C^{k+2,\al}}\le C(\frac12)^n.$ Recalling $z_n=Du_n\op Du_{n-1}$ and making $C$ larger if we need, we see that for every $n\ge0$ we have
\e\l{iii2}
\|z_n-z_{n+1}\|_{C^\al}\le \|u_n-u_{n+1}\|_{C^{2,\al}}+\|u_{n-1}-u_n\|_{C^{2,\al}}\le \frac C{2^n}.
\e
Since $p_1,\dots,p_\mu\le k$ it follows also that 
\begin{align*}
\max\{\|\nb^{p_1}z_n-\nb^{p_1}z_{n+1}\|_{C^\al},\dots,
\|\nb^{p_\mu}z_n-\nb^{p_\mu}z_{n+1}\|_{C^\al}\}
\le\|z_n-z_{n+1}\|_{C^{k,\al}}\\
\le
\|u_n-u_{n+1}\|_{C^{k+2,\al}}+\|u_{n-1}-u_n\|_{C^{k+2,\al}}
\le \frac C{2^n}.
\end{align*}
This with \eq{ii0}, \eq{iii1} and \eq{iii2} implies that making $C$ larger if we need, we have
\ea\l{iii3}
&\|F_{p_1\dots p_\mu}(z_n) (\nb^{p_1}z_n\otimes\dots\otimes\nb^{p_\mu}z_n)\\
&-F_{p_1\dots p_\mu}(z_{n+1}) (\nb^{p_1}z_{n+1}\otimes\dots\otimes\nb^{p_\mu}z_{n+1})\|_{C^\al}
\le \frac C{2^n}.
\ea
On the other hand, making $C$ larger if we need, it follows that for every $n\ge1$ we have
\[
\|G(Du_{n-1})-G(Du_{n})\|_{C^\al}\le C\|Du_{n-1}-Du_{n}\|_{C^{\al}}\le \frac{C^2}{2^n}. 
\]
Recalling that $\nb_0\cdots\nb_k Du_n$ is $C^{\al}$ bounded and making $C$ larger if we need, we get
$\|(G(Du_{n-1})-G(Du_{n}))\bullet (\nb_0\cdots\nb_k Du_n) \|_{C^\al}\le C(\frac12)^n$ and accordingly
\begin{align*}
\|G(Du_{n-1})\bullet(\nb_0\cdots\nb_k Du_{n-1})-
G(Du_{n})\bullet(\nb_0\cdots\nb_k Du_{n})\|_{C^\al}\\
\le \frac C{2^n}+ \|G(Du_n)\|_{C^\al} 
\|\nb_0\cdots\nb_k D(u_{n-1}-u_n)\|_{C^\al}\\
\le \frac C{2^n}+ \frac12 A^{-1}\|\nb_0\cdots\nb_k D(u_{n-1}-u_n)\|_{C^\al}.
\end{align*}
This with \eq{ii0} and \eq{iii3} that
\e\l{iii4}
\|\nb_0\cdots\nb_k\De( u_n-u_{n+1})\|_{C^\al}\le
\frac C{2^n}+ \|G(Du_n)\|_{C^\al} 
\|\nb_0\cdots\nb_k D(u_{n-1}-u_n)\|_{C^\al}.
\e
Generalizing \eq{i4} and making $C$ larger if we need, it follows that for every $n\ge0$ we have
\[
\|\nb_0\cdots\nb_k D (u_n-u_{n+1})\|_{C^{\al}}\le A\|\nb_0\cdots\nb_k\De(u_n-u_{n+1})\|_{C^{\al}}+C\|u_n-u_{n+1}\|_{C^{k+2,\al}}.
\]
Combining this with \eq{iii4} and making $C$ larger if we need, it follows that for every $n\ge1$ we have
\[
\|\nb_0\cdots\nb_k D (u_n-u_{n+1})\|_{C^{\al}}\le \frac C{2^n}+\frac12
\|\nb_0\cdots\nb_k D(u_{n-1}-u_n)\|_{C^\al}.
\]
Doing an induction on $n$ and making $C$ larger if we need, we see that for every $n\ge0$ we have $\|\nb_0\cdots\nb_k D (u_n-u_{n+1})\|_{C^{\al}}\le Cn(\frac12)^n.$ Recalling $\frac12<\et$ and making $C$ larger if we need, we have $\|\nb_0\cdots\nb_k D (u_n-u_{n+1})\|_{C^{\al}}\le C\et^n.$ Since $\nb_0=\nb_{\xi_0},\dots,\nb_k=\nb_{\xi_k}$ with $\xi_0,\dots,\xi_k$ arbitrary it follows that making $C$ larger if we need, we have $\|\nb^{k+1} D(u_n-u_{n+1})\|_{C^{\al}}\le C\et_{k+1}^n.$ Thus $\|u_n-u_{n+1}\|_{C^{k+3,\al}}\le C\et_{k+1}^n;$ that is, \eq{iii0} holds for $q=k+1.$ This completes the proof of Theorem \ref{5}(iii).
\end{proof}

\begin{proof}[\bf Proof of Theorem \ref{5}(iv)]
Theorem \ref{5}(iv) claims that for every $p,q=0,1,2,\dots$ the following holds: 
\begin{equation}\l{iv0}\parbox{10cm}{
there exists $B>0$ independent of $t,n$ and such that for every $n\ge0$ we have $\|\partial_t^q u_n\|_{C^{p+2,\al}}\le B.$
}\end{equation}
We prove this by a double induction with $q$ an outer parameter and $p$ an inner parameter. Theorem \ref{5}(ii) implies already that \eq{iv0} holds for $q=0.$ Let $l\ge0$ be an integer and let \eq{iv0} hold for $p\ge0,q\le l.$ We show by the inner induction that \eq{iv0} holds for $p\ge0,q=l+1.$

We show first that \eq{iv0} holds for $p=0,q=l+1.$ Applying $\partial_t$ to both sides of \eq{Deu_n} we get a smooth map $H_0:W\op W\to \R$ such that
\e\l{iv1} \partial_t\De u_{n}=H_0(z_n)+G(Du_{n-1})\bullet \partial_tDu_{n-1}.\e
Let $P$ be as in \eq{P} with $l$ in place of $k.$ Applying $\partial_t^l$ to both sides of \eq{iv1} we get a finite family $(H_{p_1\dots p_\mu}: W\op W\to (W^*)^{p_0}\otimes\dots\otimes(W^*)^{p_\mu})_{(p_1,\dots,p_\mu)\in P}$ of smooth maps and such that 
\e\l{iv2} \partial_t^{l+1}\De u_{n}=\sum_{(p_1,\dots,p_\mu)\in P}H_{p_1\dots p_\mu}(z_n)\bullet (\partial_t^{p_1}z_n\otimes\dots\otimes \partial_t^{p_\mu}z_n)+G(Du_{n-1})\bullet \partial_t^{l+1}Du_{n-1}.
\e
Let $(p_1,\dots,p_\mu)\in P.$ Since $X,\Om$ are compact and $z_n$ $C^\al$ bounded we get a constant $B>0$ independent of $t,n$ and such that for every $n\ge 0$ we have $\|H_{p_1\dots p_\mu}(z_n)\|_{C^{\al}}\le B.$ By the definition of $P$ we have $p_1,\dots,p_\mu\le l$ and so by hypothesis, making $B$ larger if we need, it follows that for every $n\ge0$ we have $\|\partial_t^{p_1}z_n\|_{C^{\al}}\le B,\dots,
\|\nb^k\partial_t^{p_\mu}z_n\|_{C^{\al}}\le B.$ These imply that for every $n\ge0$ we have
$\|H_{p_1\dots p_\mu}(z_n)\bullet (\partial_t^{p_1}z_n\otimes\dots\otimes \partial_t^{p_\mu}z_n)\|_{C^{\al}}\le B.$ Using \eq{iv2}, recalling that $P$ is finite and making $B$ larger if we need, it follows that for every $n\ge1$ we have
\e\l{iv3}
\|\De \partial_t^{l+1}  u_{n}\|_{C^{\al}}\le B+\frac12A^{-1} \|\partial_t^{l+1} Du_{n-1}\|_{C^{\al}}.
\e
On the other hand, computing the commutator $\partial_t^{l+1}D-D\partial_t^{l+1}$ we get a constant $C>0$ independent of $t,n$ and such that for every $n\ge1$ we have
\ea\l{iv4}
\|\partial_t^{l+1} Du_{n-1}\|_{C^{\al}}\le \|D\partial_t^{l+1} u_{n-1}\|_{C^{\al}}+C\sum_{q=0}^l\|\partial_t^q u_{n-1}\|_{C^{2,\al}}\\
\le \|\partial_t^{l+1} u_{n-1}\|_{C^{2,\al}}+B,
\ea
using the outer induction hypothesis and making $B$ larger if we need. Since $\int_X u_n{\rm dvol}=0$ it follows that $\int_X \partial_t^lu_n{\rm dvol}=0.$ Corollary \ref{cor: Schaud} implies therefore that $\|\partial_t^{l+1}u_n\|_{C^{2,\al}}\le A\|\De\partial_t^{l+1}u_n\|_{C^{\al}}.$ Combining this with \eq{iv3} and \eq{iv4}, and making $B$ larger if we need, we see that for every $n\ge1$ we have
\[
\|\partial_t^{l+1} u_n\|_{C^{2,\al}}\le B+\frac12\|\partial_t^{l+1}u_{n-1}\|_{C^{2,\al}}.
\]
So $\|\partial_t^{l+1}u_{n}\|_{C^{2,\al}}$ is bounded with respect to $n.$ Thus \eq{iv0} holds for $p=0,q=l+1.$

Suppose next that $k\ge0$ is an integer and that \eq{iv0} holds for $p\le k,q\le l+1.$ Let $\xi_0,\dots,\xi_k$ be vector fields on $X$ which are independent of $t,n.$ Applying $\nb_0\cdots\nb_k$ to both sides of \eq{iv2} we get a finite family $(E_{p_1\dots p_\mu}: W\op W\to (TX)^{\ot \mu(k+1)}\otimes(W^*)^{p_0}\otimes\dots\otimes(W^*)^{p_\mu})_{(p_1,\dots,p_\mu)\in P}$ of smooth maps such that
\ea\l{iv5} \De \nb_0\cdots\nb_k\partial_t^{l+1} u_{n}=\sum_{(p_1,\dots,p_\mu)\in P}E_{p_1\dots p_\mu}(z_n)\bullet (\nb^{k+1}\partial_t^{p_1}z_n\otimes\dots\otimes \nb^{k+1}\partial_t^{p_\mu}z_n)\\
+G(Du_{n-1})\bullet \nb_0\cdots\nb_k\partial_t^{l+1}Du_{n-1}.
\ea
Let $(p_1,\dots,p_\mu)\in P.$ Since $X,\Om$ are compact and $z_n$ $C^\al$ bounded we get a constant $B>0$ independent of $t,n$ and such that for every $n\ge 0$ we have $\|E_{p_1\dots p_\mu}(z_n)\|_{C^{\al}}\le B.$ Making $B$ larger if we need, it follows that for every $n\ge0$ we have \[\|\nb^{k+1}\partial_t^{p_1}z_n\|_{C^{\al}}\le B,\dots,
\|\nb^{k+1}\partial_t^{p_\mu}z_n\|_{C^{\al}}\le B.\]
 These imply that for every $n\ge0$ we have
$\|E_{p_1\dots p_\mu}(z_n)\bullet (\nb^{k+1}\partial_t^{p_1}z_n\otimes\dots\otimes \nb^{k+1}\partial_t^{p_\mu}z_n)\|_{C^{\al}}\le B.$ Using \eq{iv5}, recalling that $P$ is finite and making $B$ larger if we need, it follows that for every $n\ge1$ we have
\e\l{iv6}
\|\De \nb_0\cdots\nb_k\partial_t^{l+1}  u_{n}\|_{C^{\al}}\le B+\frac12 A^{-1} \|\nb_0\cdots\nb_k\partial_t^{l+1} Du_{n-1}\|_{C^{\al}}.
\e
On the other hand, computing the commutator $\nb_0\cdots\nb_k\partial_t^{l+1}D-D\nb_0\cdots\nb_k\partial_t^{l+1}$ we get a constant $C>0$ independent of $t,n$ and such that for every $n\ge1$ we have
\ea\l{iv7}
\|\nb_0\cdots\nb_k\partial_t^{l+1} Du_n\|_{C^{\al}}\le \|D\nb_0\cdots\nb_k \partial_t^{l+1} u_n\|_{C^{\al}}+C\sum_{p+q\le 1+k+l}\|\partial_t^q u_n\|_{C^{p+2,\al}}\\
\le \|\nb_0\cdots\nb_k \partial_t^{l+1} u_n\|_{C^{2,\al}}+B,
\ea
using the outer and inner induction hypotheses and making $B$ larger if we need.
Recall now from Theorem \ref{thm: Schaud} that
\ea\l{iv8}
\|\nb_0\cdots\nb_k\partial_t^{l+1}u_n\|_{C^{2,\al}}\le A(\|\De \nb_0\cdots\nb_k\partial_t^{l+1}u_n\|_{C^{\al}}
+\|\nb_0\cdots\nb_k\partial_t^{l+1}u_n\|_{C^0})\\
\le A\|\De \nb_0\cdots\nb_k\partial_t^{l+1}u_n\|_{C^{\al}}
+B,
\ea
using the inner induction hypothesis and making $B$ larger if we need. Combining \eq{iv6}--\eq{iv8} and making $B$ larger if we need, we see that for every $n\ge1$ we have
\[
\|\nb_0\cdots\nb_k\partial_t^{l+1} u_n\|_{C^{2,\al}}\le B+\frac12\|\partial_t^{l+1}\nb_0\cdots\nb_k u_n\|_{C^{2,\al}}.
\]
So $\|\nb_0\cdots\nb_k\partial_t^{l+1} u_{n}\|_{C^{2,\al}}$ is bounded with respect to $n.$ As $\xi_0,\dots,\xi_k$ has been arbitrary vector fields this implies that the full derivative $\nb^{k+1}\partial_t^{l+1}u_n$ is $C^{2,\al}$ bounded with respect to $n.$ Thus $\partial_t^{l+1}u_n$ is $C^{k+3,\al}$ bounded with respect to $n;$ that is, \eq{iv0} holds for $p=k+1,q=l+1.$ This completes the inner induction and \eq{iv0} holds for $p\ge0,q=l+1.$ The latter completes the outer induction and \eq{iv0} holds for every $p,q\ge0.$ 
\end{proof}

\begin{proof}[\bf Proof of Theorem \ref{5}(v)]
We prove that for $p,q=0,1,2,\dots$ the following holds:
\begin{equation}\l{v-1}\parbox{10cm}{
for every $\ze\in(\frac12,1)$ there exists $C>0$ independent of $t,n$ and such that for every $n\ge0$ we have $\|\partial_t^q( u_n-u_{n+1})\|_{C^{p+2,\al}}\le C\ze^n.$ 
}\end{equation}
We do this by a double induction with $q$ an outer parameter and $p$ an inner parameter. Theorem \ref{5}(iii) implies already that \eq{v-1} holds for $q=0.$ Let $l\ge0$ be an integer and let \eq{v-1} hold for $p\ge0,q\le l.$ We show by the inner induction that \eq{v-1} holds for $p\ge0,q=l+1.$

We show first that \eq{v0} holds for $p=0,q=l+1.$ Fix $\ze,\et$ with $\frac12<\et<\ze<1.$ Let $P$ and $(H_{p_1\dots p_\mu}: W\op W\to (W^*)^{p_0}\otimes\dots\otimes(W^*)^{p_\mu})_{(p_1,\dots,p_\mu)\in P}$ be as in the proof of Theorem \ref{5}(iv). Let $(p_1,\dots,p_\mu)\in P.$ Since $X,\Om$ are compact and $(z_n)_{n=0}^\iy$ $C^\al$ bounded we get a constant $C>0$ independent of $t,n$ and such that $n\ge1$ we have
\ea\l{v0}
\|H_{p_1\dots p_\mu}(z_n)-H_{p_1\dots p_\mu}(z_{n+1})\|_{C^{\al}}\le C\|z_n-z_{n+1}\|_{C^{\al}}\\
\le C\|Du_n-Du_{n+1}\|_{C^{\al}}+C\|Du_{n-1}-Du_{n}\|_{C^{\al}} 
\le 2C^2\et^n,
\ea
using the outer induction hypothesis and making $C$ larger if we need.
Since $p_1,\dots,p_\mu\le l$ it follows by the outer induction hypothesis that   
$\|\partial_t^{p_1}(z_n-z_{n+1})\|_{C^{\al}}\le C\et^n,\dots,
\|\partial_t^{p_\mu}(z_n-z_{n+1})\|_{C^{\al}}\le C\et^n.$ Combining these with \eq{v0} and making $C$ larger if we need, we find that for every $(p_1,\dots,p_\mu)\in P$ we have
\ea\l{v1}
&\|H_{p_1\dots p_\mu}(z_n)\bullet (\partial_t^{p_1}z_n\otimes\dots\otimes \partial_t^{p_\mu}z_n)\\
&-H_{p_1\dots p_\mu}(z_{n+1})\bullet (\partial_t^{p_1}z_{n+1}\otimes\dots\otimes \partial_t^{p_\mu}z_{n+1})\|_{C^{\al}}\le C\et^n.
\ea
Making $C$ larger if we need, it follows also that for every $n\ge 1$ we have
\e\l{v2}
\|G(Du_{n-1})-G(Du_n)\|_{C^{\al}}\le C\|Du_{n-1}-Du_n\|_{C^{\al}}\le C\et^n.
\e
Using \eq{v0}, recalling that $P$ is finite and making $C$ larger if we need, it follows that for every $n\ge1$ we have
\e\l{v3}
\|\De \partial_t^{l+1}  (u_{n}-u_{n+1}\|_{C^{\al}}\le C\et^n+\frac12A^{-1} \|\partial_t^{l+1} D(u_{n-1}-u_n)\|_{C^{\al}}.
\e
On the other hand, computing the commutator $\partial_t^{l+1}D-D\partial_t^{l+1}$ we get a constant $C>0$ independent of $t,n$ and such that for every $n\ge1$ we have
\ea\l{v4}
\|\partial_t^{l+1} D(u_{n-1}-u_n)\|_{C^{\al}}\le \|D\partial_t^{l+1} (u_{n-1}-u_n)\|_{C^{\al}}+C\sum_{q=0}^l\|\partial_t^q (u_{n-1}-u_n)\|_{C^{2,\al}}\\
\le \|\partial_t^{l+1} (u_{n-1}-u_n)\|_{C^{2,\al}}+C^2\et^n,
\ea
using the outer induction hypothesis and making $C$ larger if we need. Corollary \ref{cor: Schaud} implies now that $\|\partial_t^{l+1}(u_n-u_{n+1})\|_{C^{2,\al}}\le A\|\De\partial_t^{l+1}(u_n-u_{n+1})\|_{C^{\al}}.$ Combining this with \eq{v3} and \eq{v4}, and making $C$ larger if we need, we see that for every $n\ge1$ we have
\[
\|\partial_t^{l+1} (u_n-u_{n+1})\|_{C^{2,\al}}\le C\et^n+\frac12\|\partial_t^{l+1}(u_n-u_{n-1})\|_{C^{2,\al}}.
\]
Doing an induction on $n$ and making $C$ larger if we need, we find that for every $n\ge1$ we have $\|\partial_t^{l+1}u_{n}\|_{C^{2,\al}}\le Cn\et^n.$ Recalling $\ze>\et$ and making $C$ larger if we need, we see that for every $n\ge1$ we have  $\|\partial_t^{l+1}u_{n}\|_{C^{2,\al}}\le C\ze^n.$ Thus \eq{v-1} holds for $p=0,q=l+1.$

Suppose next that $k\ge0$ is an integer and that \eq{v-1} holds for $p\le k,q\le l+1.$ Fix again $\ze,\et$ with $\ep<\et<\ze<1.$ Let $\xi_0,\dots,\xi_k$ be vector fields on $X$ which are independent of $t,n.$ Let $(p_1,\dots,p_\mu)\in P$ and let $E_{p_1\dots p_\mu}$ be as in \eq{iv5}. Since $X,\Om$ are compact and $u_n$ $C^\al$ bounded we get a constant $C>0$ independent of $t,n$ and such that for every $n\ge 0$ we have
\e\l{v5}
\|E_{p_1\dots p_\mu}(z_n)-E_{p_1\dots p_\mu}(z_{n+1})\|_{C^{\al}}\le C\|z_n-z_{n+1}\|_{C^\al} \le C^2\et^n,
\e
using the outer induction hypothesis and making $C$ larger if we need. Recalling $p_1,\dots,p_\mu\le l$ and making $C$ larger if we need, it follows that for every $n\ge0$ we have
\e\l{v6}
\|\nb^{k+1}\partial_t^{p_1}(z_n-z_{n+1})\|_{C^{\al}}\le C\et^n,\dots,
\|\nb^{k+1}\partial_t^{p_\mu}(z_n-z_{n+1})\|_{C^{\al}}\le C\et^n.
\e
These imply that for every $n\ge0$ we have
\ea\l{v7}
&\|E_{p_1\dots p_\mu}(z_n)\bullet (\nb^{k+1}\partial_t^{p_1}z_n\otimes\dots\otimes \nb^{k+1}\partial_t^{p_\mu}z_n)\\
&-E_{p_1\dots p_\mu}(z_{n+1})\bullet (\nb^{k+1}\partial_t^{p_1}z_n\otimes\dots\otimes \nb^{k+1}\partial_t^{p_\mu}z_{n+1})
\|_{C^{\al}}\le C\et^n.
\ea
Combining \eq{iv5} with \eq{v5}--\eq{v7}, recalling that $P$ is finite and making $C$ larger if we need, it follows that for every $n\ge1$ we have
\e\l{v8}
\|\De \nb_0\cdots\nb_k\partial_t^{l+1} ( u_{n}-u_{n+1})\|_{C^{\al}}\le C\et^n+A^{-1}\frac12 \|\nb_0\cdots\nb_k\partial_t^{l+1} D(u_{n-1}-u_n)\|_{C^{\al}}.
\e
On the other hand, computing the commutator $\nb_0\cdots\nb_k\partial_t^{l+1}D-D\nb_0\cdots\nb_k\partial_t^{l+1}$ we get a constant $C>0$ independent of $t,n$ and such that for every $n\ge1$ we have
\ea\l{v9}
&\|\nb_0\cdots\nb_k\partial_t^{l+1} D(u_{n-1}-u_{n})\|_{C^{\al}}\\
&\le \|D\nb_0\cdots\nb_k \partial_t^{l+1} (u_{n-1}-u_{n})\|_{C^{\al}}+C\sum_{p+q\le 1+k+l} \|\partial_t^q (u_{n-1}-u_{n})\|_{C^{p+2,\al}}\\
&\le \|\nb_0\cdots\nb_k \partial_t^{l+1} (u_{n-1}-u_{n})\|_{C^{2,\al}}+C\et^n,
\ea
using the outer and inner induction hypotheses, and making $C$ larger if we need.
Recall now from Theorem \ref{thm: Schaud} that
\ea\l{v10}
&\|\nb_0\cdots\nb_k\partial_t^{l+1}(u_{n-1}-u_{n})\|_{C^{2,\al}}\\
&\le A(\|\De \nb_0\cdots\nb_k\partial_t^{l+1}(u_{n-1}-u_{n})\|_{C^{\al}}
+\|\nb_0\cdots\nb_k\partial_t^{l+1}(u_{n-1}-u_{n})\|_{C^0})\\
&\le A\|\De \nb_0\cdots\nb_k\partial_t^{l+1}(u_{n-1}-u_{n})\|_{C^{\al}}+C\et^n,
\ea
using the inner induction hypothesis and making $C$ larger if we need. Combining \eq{v8}--\eq{v10} and making $C$ larger if we need, we see that for every $n\ge1$ we have
\[
\|\nb_0\cdots\nb_k\partial_t^{l+1} (u_n-u_{n+1})\|_{C^{2,\al}}\le C\et^n+\frac12\|\partial_t^{l+1}\nb_0\cdots\nb_k (u_{n-1}-u_{n})\|_{C^{2,\al}}.
\]
Doing an induction on $n$ and making $C$ larger if we need, we see that for every $n\ge0$ we have $\|\nb_0\cdots\nb_k\partial_t^{l+1} (u_n-u_{n+1})\|_{C^{2,\al}}\le Cn\et^n.$ Recalling $\ze>\et$ and making $C$ larger if we need, we find that for every $n\ge0$ we have $\|\nb_0\cdots\nb_k\partial_t^{l+1} (u_n-u_{n+1})\|_{C^{2,\al}}\le C\ze^n.$ Noting that $\xi_0,\dots,\xi_k$ has been arbitrary vector fields and making $C$ larger if we need, we see that the full derivative $\nb^{k+1}\partial_t^{l+1}u_n$ satisfies the estimate $\|\nb^{k+1}\partial_t^{l+1} (u_n-u_{n+1})\|_{C^{2,\al}}\le C\ze^n.$ Thus $\|\partial_t^{l+1} (u_n-u_{n+1})\|_{C^{k+3,\al}}\le C\ze^n;$ that is, \eq{v-1} holds for $p=k+1,q=l+1.$ This completes the inner induction and \eq{v-1} holds for $p\ge0,q=l+1.$ The latter completes the outer induction and \eq{v-1} holds for every $p,q\ge0.$ 
\end{proof}

This completes the proof of Theorem \ref{5}.
\end{proof}

The following is an immediate consequence of Theorem \ref{5}.
\begin{cor}\l{cor 5}
Let $\Om\sb\R^d$ be a compact set whose interior $\Om^\cm$ is non-empty. 
Let $\ep>0,(\la_n\in C^\iy(\R^d\times X))_{n=0}^\iy$ and $(u_n\in C^\iy(\R^d\times X))_{n=0}^\iy$ be as in Theorem \ref{5}. Then for every $(t,x)\in\Om\times X$ the limit $u_\iy(t,x)=\lim_{n\to\iy}u_n(t,x)$ exists and defines a $C^\iy$ function from $\Om^\cm\times X$ to $\R.$
\end{cor}
\begin{proof}
The hypotheses of Theorem \ref{5} imply that the pointwise limit $u_\iy(t,x)=\lim_{n\to\iy}u_n(t,x)$ exists for $(t,x)\in\Om\times X$ and defines a $C^0$ function from $\Om\times X$ to $\R.$ We show that for $q=0,1,2,\dots$ the following holds:
\e\l{c50}\parbox{10cm}{
the $t$-derivative $\partial_t^{l}u_\iy:\Om^\cm\times X\to\R$ exists and the sequence $(\partial_t^qu_n)_{n=0}^\iy$ converges uniformly on $\Om^\cm\times X$ to $\partial_t^qu_\iy.$
}\e
We know already that \eq{c50} holds for $q=0.$ Let $l\ge1$ and let \eq{c50} hold for $q=l-1.$ By Theorem \ref{5}(v) the sequence $(\partial_t^lu_n)_{n=0}^\iy$ converges uniformly. So $\partial_t^lu_\iy$ exists and is equal to the limit $\lim_{n\to\iy} \partial_t^lu_n.$ Thus \eq{c50} holds for $q=l$ and the induction is complete. 

Fix $l\in\{0,1,2,\dots\}.$ We show that for $q=0,1,2,\dots$ the following holds:
\e\l{c51}\parbox{10cm}{
$\nb^q\partial_t^{l}u_\iy:\Om^\cm\times X\to\R$ exists and the sequence $(\nb^q\partial_t^lu_n)_{n=0}^\iy$ converges uniformly on $\Om^\cm\times X$ to $\nb^q\partial_t^lu_\iy.$
}\e
We know already that \eq{c51} holds for $q=0.$ Let $k\ge1$ and let \eq{c51} hold for $q=k-1.$ Theorem \ref{5}(v) implies that the sequence $(\nb^k\partial_t^lu_n)_{n=0}^\iy$ converges uniformly. So $\nb^k\partial_t^lu_\iy$ exists and is equal to the limit $\lim_{n\to\iy} \nb^k\partial_t^lu_n.$ Thus \eq{c51} holds for $q=k$ and the induction is complete. In particular, the derivative $\nb^k\partial_t^{l}u_\iy:\Om^\cm\times X\to\R$ exists and is $C^0$ for every $k,l\ge0$ as we have to prove.
\end{proof}
\begin{rmk}\label{Ascoli}
Write $\Om^\cm$ as a countable union of compact subsets. By Theorem \ref{5}(iv) we can apply Arzel\`a--Ascoli's theorem to each of these compact subsets. By the diagonal argument therefore there is a subsequence of $(u_n)_{n=0}^\iy$ which converges in the $C^\iy$ sense on every compact subset of $\Om^\cm.$ This will be sufficient for the application in \S\ref{sect11}.

We have proved Theorem \ref{5}(iii),(v) for the sake of completeness. Their advantage is that we do not have to take subsequences as above.
\end{rmk}

\section{Proof of Theorem \ref{smooth dependence}} \l{sect11}
In this section we prove that the first approximate solutions produced in \S\ref{app sec}, the second approximate solutions produced in \S\ref{pert4} and the true solutions produced in \S\ref{Proof of 1} depend smoothly on $s,t.$ These three families of solutions are all of the form $(\io^{st}:N\to M)_{(s,t)\in\cU}$ with $\cU\sb\R^n\times(0,\iy)$ an open subset. Each $\io^{st}$ is a smooth map between the manifolds $N,M$ which are independent of $s,t.$ 

Note that as $\cU\sb\R^n\times(0,\iy)$ is an open set, its generic points do not satisfy the conditions (i),(ii) of Theorem \ref{obst thm}. This means that if $(s,t)$ is generic then the smooth map $\io^{st}:N\to M$ has no particular geometric meaning. In other words, $\cU$ is an auxiliary space which we use only to prove that $\io^{st}$ depends smoothly on $s,t.$

It is clear from the construction of \S\ref{app sec} that the first approximate solutions depend smoothly on $s,t.$ As we have just mentioned, we no longer require $(s,t)$ to satisfy the condition (i) in Theorem \ref{obst thm}; and accordingly, the corresponding map $\io^{st}:N\to (M,\om^s)$ need no longer be a Lagrangian embedding.

As in \S\ref{pert2} choose the constant $\th>2-2\nu.$ As in Corollary \ref{sol cor}, for each $(\bar s,\bar t)\in\cG$ with $|\bar s|<\bar t^\th$ we perturb the first approximate solution $\io^{\bar s\bar t}:N\to M$ to the second approximate solution. Choose an open neighbourhood $\cU\in \R^d\times(0,\iy)$ of $(\bar s,\bar t)\in\cG$ and apply the proof of Corollary \ref{sol cor} to every $(s,t)\in\cU.$ Then the second approximate solutions are parametrized by $(s,t)$ close enough to $(\bar s,\bar t).$ We prove
\begin{prop}\l{prop 11.1}
The second approximate solutions depend smoothly on $(s,t)\in\cU.$
\end{prop}
\begin{proof}
As we have just mentioned, we apply the proof of Corollary \ref{sol cor} to $(s,t)\in\cU.$ So there exists a corresponding sequence $(v_n\in C^{1/2}_{\nu-2}(N,\R))_{n=0}^\iy$ with each $v_n$ depending on $(s,t)\in\cU.$ We do not know yet whether this is smooth with respect to $(s,t)\in\cU.$ The sequence $(v_n)_{n=0}^\iy$ satisfies $v_0=0$ and for $n\ge1$ the recursive equation $v_n=-P0-Q\d Rv_{n-1}.$ Here $P,Q$ and $R$ also depend on $(s,t)\in\cU.$ But these are made from the geometric data and therefore smooth with respect to $(s,t)\in\cU.$

As in Corollary \ref{sol cor}, for $n=0,1,2,\dots$ define $u_n\op w_n\in C^{5/2}_\nu(N,\R)\op H^0(X,\R)$ by $u_n\op w_n:=(R\op\pi)v_n.$ Put $\la_n:=-P0-Ew_n.$ Then for $n\ge1$ the equation $v_n=-P0-Q\d Rv_{n-1}$ implies
\e\l{De u_n}\De u_n=\la_n-Q\d u_{n-1}.\e 
We prove by induction on $n$ that $u_n\in C^\iy(N,\R).$ For $n=0,$ as $v_0=0$ we have $u_0=0\in C^\iy(N,\R).$ For $n\ge1,$ if $u_{n-1}\in C^\iy(N,\R)$ then applying the elliptic regularity theorem to \eq{De u_n} we find that $u_n\in C^\iy(N,\R).$

For each $(s,t)\in\cU$ the sequence $(v_n)_{n=0}^\iy$ converges in $C^{1/2}_{\nu-2}(N,\R)$ and so the sequence $(u_n)_{n=0}^\iy$ converges in $C^{1/2}_{\nu-2}(N,\R).$ Moreover, since $Q$ is the nearly quadratic term of the Taylor expansion at $0$ of $P$ it follows that \eq{De u_n} may be rewritten in the form \eq{Deu_n}. 

By Theorem \ref{3}, if we vary the parameter $(s,t)\in\cU$ then every $u_n$ defines a $C^\iy$ function from $\cU\times X$ to $\R.$ (We think of $(s,t)$ as a single parameter $t$ when we apply the results of \S\ref{sect10}.) We can therefore apply Corollary \ref{cor 5} to the sequence $(u_n)_{n=0}^\iy.$ As a result of this, the limit $u:=\lim_{n\to\iy}u_n$ is $C^\iy$ with respect to $s,t.$

Since $(v_n)_{n=0}^\iy$ converges in $C^{1/2}_{\nu-2}(N,\R)$ it follows also that $(w_n)_{n=0}^\iy$ converges in $H^0(X,\R).$ As in the proof of Corollary \ref{sol cor} the limit $w=\lim_{n\to\iy}w_n$ satisfies the equation $P\d u+Ew=0.$ By the definition of $E,$ if we write $w=(w_Y\in\R)_{Y\in \pi_0X}$ then $Ew=\sum_{Y\in\pi_0 X} E_Yw_Y$ where each $E_Y:Y\to\R$ is a compactly supported smooth function independent of $s,t.$ Since $P\d u$ is already $C^\iy$ with respect $s,t$ it follows that so is $(w_Y\in\R)_{Y\in \pi_0X}.$
\end{proof}



We make now a version of Theorem \ref{5.3} which applies to the second approximate solutions parametrized by the open subset of the parameter space $\R^d\times(0,\iy).$ The major change we need is about the condition (vi) of Theorem \ref{5.3}. This part corresponds to the condition $(s,t)\in\cG$ when we apply Theorem \ref{5.3} to the second approximate solutions. But now the second approximate solutions are parametrized by the larger open subset and the condition $(s,t)\in\cG$ will not always hold. We introduce therefore an auxiliary term $\Xi$ as in the following statement.

\begin{thm}\l{thm 11.2}
For every $m\in\{3,4,5,\dots\},c>0$ and $\ep>0$ there exists $t_1>0$ such that if $t\in(0,t_1)$ then the following holds. Let $(M;\om,J,\Om)$ and $\ps:M\to(0,\iy)$ be as in Theorem \ref{5.3}. Let $N\sb (M,\om)$ be a compact oriented submanifold; we do not suppose that $N$ is Lagrangian as in Theorem \ref{5.3} and nor do we suppose $\int_N\Im\Om=0.$  Let $g$ be as in Theorem \ref{5.3} and suppose that 
\iz
\item[\bf(a)] the condition {\rm (i)} of Theorem \ref{5.3} holds with $2c$ in place of $c.$ 
\iz
Let $U\sb M$ be as in Theorem \ref{5.3} except that it is no longer a Weinstein neighbourhood in the ordinary sense. Suppose that
\iz\item[\bf(b)]
the condition {\rm (ii)} of Theorem \ref{5.3} holds with $2c$ in place of $c.$ 
\iz
Let the Sobolev spaces and the finite-dimensional space $W$ be as in Theorem \ref{5.3}. Suppose that
\iz
\item[\bf(c)] 
the conditions {\rm (iii)--(v)} of Theorem \ref{5.3} hold with $2c$ in place of $c.$ 
\iz
Let {\rm dvol} and $e^{i\Th}:N\to S^1$ be as in Theorem \ref{5.3}. Let $\Xi:N\to\R$ be a smooth function with $\int_{N}\Xi\,{\rm dvol}=\int_{N}\Im\Om$ and such that
\iz
\item[\bf(d)]
the condition {\rm (vi),(vii)} of Theorem \ref{5.3} hold with $2c$ in place of $c$ and with $(\ps^m\sin\Th)-\Xi$ in place of $\ps^m\sin\Th.$
\iz
Then there exists on $N$ an exact $1$-form $\xi$ whose graph lies in the neighbourhood $U\sb(M,\om)$ and such that $\Im\Om|_{\gr\xi}=\Xi{\rm dvol}$ under the graphical diffeomorphism $N\cong\gr\xi.$ There is moreover a constant $c'>0$ depending only on $m,c,\ep$ such that $\sup_{N}|\xi|\le c't^\ep$ where $|\,\,|$ is computed pointwise with respect to $g.$
\end{thm} 
\begin{proof}
We modify the proof of Theorem \ref{5.3} (which goes back to \cite[Theorem 5.3]{J3}). For a smooth function $\ph:N\to\R$ define $A\ph:N\to\R$ by $(A\ph){\rm dvol}:= \Im\Om|_{\gr\d\ph}-\Xi{\rm dvol}.$ We solve the equation $A\ph=0.$ Let $\De$ be the differential of $A$ at $\ph=0.$ Then $\De \ph=-\d^*((\ps^m\sin\Th)\ph),$ which defines a bijection from the set of $\ph$ with $\int_N\ph{\rm dvol}=0$ to the same set. Define $\ph\mapsto B\ph$ by $A\ph=A0+\De u+B\ph.$ Then $A\ph=0$ is equivalent to $\De \ph=-A0-B\ph.$ Since $\int_{N}\Xi\,{\rm dvol}=\int_{N}\Im\Om$ it follows that $\int_NA\ph\,{\rm dvol}=0.$ On the other hand, we have always $\int_N\De\ph\,{\rm dvol}=0$ and so $\int_N (A0+B\ph){\rm dvol}=0.$ We can therefore define the map $\ph\mapsto-\De^{-1}(A0+B\ph),$ to which we apply the contraction mapping theorem.

In the circumstances of Theorem \ref{5.3} we define the same operator $A$ without the $\Xi$ term, and $A0$ is then equal to $\ps^m\sin\Th.$ Theorem \ref{5.3}(i)--(v) imply that $B\ph$ satisfies the estimates we need for the contraction mapping theorem. Theorem \ref{5.3}(vi),(vii) imply that $\ps^m\sin\Th=A0$ satisfies the estimates we need for the contraction mapping theorem. The details are in the proof of \cite[Theorem 5.3]{J3}. 

In the current circumstances the conditions (a)--(c) imply that $B\ph$ satisfies the estimates we need for the contraction mapping theorem, and (d) implies that $(\ps^m\sin\Th)-\Xi=A0$ satisfies the estimates we need for the contraction mapping theorem. We can therefore use again the contraction mapping theorem, completing the proof. 
\end{proof}

As in Proposition \ref{prop 11.1} fix $(\bar s,\bar t)\in\cG$ and let $(s,t)\in\R^d\times(0,\iy)$ be close enough to $(\bar s,\bar t).$ Denote by $N^{st}\sb M$ the second approximate solution corresponding to the nearby $(s,t).$ By Proposition \ref{prop 11.1} this $N^{st}$ is smooth with respect to $s,t.$ 

As in the proof of Theorem \ref{1} we can suppose that the conditions (i)--(vii) of Theorem \ref{5.3} hold with $N^{\bar s\bar t}$ in place of $N.$ Then there exists a neighbourhood $\cU$ of $(\bar s,\bar t)\in\R^n\times(0,\iy)$ such that if $(s,t)\in\cU$ then the conditions (a)--(c) of Theorem \ref{thm 11.2} hold with $N^{st}$ in place of $N.$

For each connected component $Y\sb X $ let $E_Y:N^{st}\to\R$ be a compactly supported smooth function which is supported in the region diffeomorphic to a compact subset of $Y,$ and such that $\int_{N} E_Y{\rm dvol}=1.$ So $E_Y$ is essentially independent of $s,t.$ Define $\Xi^{st}:N^{st}\to \R$ by setting
\e\l{Xi}
\Xi^{st}:=\sum_{Y\in\pi_0X}\Bigl([\bar Y]\cdot [\Im\Om^s]+\sum_{C_{xj}\sb Y}t^m\ps^s(x)^m(1+\ga_x)[C_{xj}\cap S^{2m-1}]\cdot Z(L_x)\Bigr)E_Y.
\e  
Notice that for each $x\in X^\sing,$ if we denote by $\hat L_x\sb L_x$ the complement of ends then as $L_x$ is special Lagrangian, we have $\int_{\hat L_x}\Im\Om'=0.$ By the definition of $Z(L_x)$ this is equivalent to saying that $\sum_{C_{xj}\sb C_x}[C_{xj}\cap S^{2m-1}]\cdot Z(L_x)=0.$ If we vary $x\in X^\sing$ and take the sum over $x$ then we obtain the equality
\[\sum_{Y\in\pi_0X}\sum_{C_{xj}\sb Y}t^m\ps^s(x)^m(1+\ga_x)[C_{xj}\cap S^{2m-1}]\cdot Z(L_x)=0.\] 
Hence integrating \eq{Xi} we see that $\int_N\Xi\,{\rm dvol}= \int_N \Im\Om^s.$ So the equation before Theorem \ref{thm 11.2}(d) holds with $N^{st},\Xi^{st}$ in place of $N,\Xi$ respectively.  

We verify now that Theorem \ref{thm 11.2}(d) holds in the present context. We begin with the estimate concerning Theorem \ref{5.3}(vi). Going back to the proof of Joyce \cite[Theorem 7.9]{J4} and using \eq{Xi} we see that we can do the same computation as in the proof of \cite[Proposition 7.7]{J4}. Hence we get the estimate we want. 

We turn next to the estimate concerning Theorem \ref{5.3}(vii). Since $\Xi$ is continuous with respect to $s,t$ we get for $p=1,\iy$ a neighbourhood $\cU_p$ of $(\bar s,\bar t)\in\R^n\times(0,\iy)$ such that every $(s,t)\in \cU_p$ satisfies Theorem \ref{5.3}(vii) with this $p,$ with $2c$ in place of $c$ and with $(\ps^m\sin\Th)-\Xi$ in place of $\ps^m\sin\Th.$ For $( s, t)\in \cU_1\cap \cU_\iy$ the same statement holds for $p=1,\iy.$ The interpolation inequality implies then that it holds for every $p\in[1,\iy].$ 

We can thus apply Theorem \ref{thm 11.2} and so for $t$ small enough there exists an exact $1$-form $\xi$ with $\Im\Om|_{\gr\xi}=\Xi^{st}{\rm dvol}.$ Since Theorem \ref{thm 11.2} is a direct generalization of Theorem \ref{5.3} it follows that for $t$ small enough the current $1$-form $\xi$ is an extension of that in the proof of Theorem \ref{1}. In other words, the latter $1$-form extends to an open neighbourhood $\cU$ of $(s,t)\in\R^n\times(0,\iy);$ and accordingly, the special Lagrangian submanifolds produced in Theorem \ref{1} extend to a family of submanifolds parametrized by the same open set $\cU.$   

We show finally that these submanifolds depend smoothly on $(s, t)\in\cU.$ The equation we solve in the proof of Theorem \ref{thm 11.2} is of the form $\De u=-P0-Bu,$ and $Bu$ is the nearly quadratic term. This implies that the sequence of functions provided by the contraction mapping theorem satisfies the hypotheses of Corollary \ref{cor 5}. This completes the proof.

Institute of Mathematical Sciences, ShanghaiTech University, 393 Middle Huaxia Road, Pudong New District, Shanghai, China 

e-mail address: yosukeimagi@shanghaitech.edu.cn

\end{document}